\documentclass[twoside,12pt]{article}
\usepackage{amsmath,amssymb,amsthm,longtable,graphicx}
\usepackage{txfonts,multirow}
\usepackage[alphabetic,abbrev]{amsrefs}

\newsavebox{\circlebox}
\savebox{\circlebox}{\fontencoding{OMS}\selectfont\Large\char13}
\newlength{\circleboxwdht}
\newcommand{\centercircle}[1]{%
  \setlength{\circleboxwdht}{\wd\circlebox}%
  \addtolength{\circleboxwdht}{\dp\circlebox}%
  \raisebox{0.4\dp\circlebox}{%
    \parbox[][\circleboxwdht][c]{\wd\circlebox}{\centering#1}}%
  \llap{\usebox{\circlebox}}%
}

\newtheorem{definition}{Definition}[section]
\newtheorem{theorem}[definition]{Theorem}
\newtheorem{proposition}[definition]{Proposition}
\newtheorem{example}[definition]{Example}
\newtheorem{corollary}[definition]{Corollary}
\newtheorem{lemma}[definition]{Lemma}

\newcommand{\Bf}[1]{\boldsymbol{#1}}
\newcommand{\transpose}[1]{{}^{t}#1}

\newcommand{\dis}[1]{\displaystyle{#1}}

\begin{document}

\title{On a higher level extension of Leclerc-Thibon product theorem in $q$-deformed Fock spaces}
\author{Kazuto Iijima \\
{ \footnotesize Graduate School of Mathematics, Nagoya University, Chikusa-ku, Nagoya 464-8602, Japan } \\
{ \footnotesize kazuto.iijima@math.nagoya-u.ac.jp } }

\date{}
\maketitle

\thispagestyle{empty}

\begin{abstract}
The $q$-deformed Fock spaces of higher levels were introduced by Jimbo-Misra-Miwa-Okado. 
Uglov defined a canonical bases in $q$-deformed Fock spaces of higher levels.
Leclerc-Thibon showed a product theorem in $q$-deformed Fock spaces of level one. 
The product theorem is regarded as a formal $q$-analogue of the tensor product theorem of level one. 
In this paper, we show a higher level analogue of Leclerc-Thibon product theorem under a suitable multi charge condition. 
\end{abstract}

%
%

\section{Introduction}

The $q$-deformed Fock spaces of higher levels were introduced by Jimbo-Misra-Miwa-Okado \cite{JMMO}. 
For integers $n \ge 2$, $\ell \ge 1$ and  
	a multi charge $\boldsymbol{s} = (s_{1} , \ldots , s_{\ell}) \in \mathbb{Z}^{\ell}$, 
	the $q$-deformed Fock space $\boldsymbol{F}_{q}[\boldsymbol{s}]$ of level $\ell$ is 
	the $\mathbb{Q}(q)$-vector space whose basis are indexed by $\ell$-tuples of Young diagrams. i.e. 
	$\{ | \boldsymbol{\lambda} ; \boldsymbol{s} \rangle \, | \, \boldsymbol{\lambda} \in \Pi^{\ell} \}$, 
	where $\Pi$ is the set of Young diagrams. 

The Fock space $\Bf{F}_{q}[\Bf{s}]$ is endowed with the action of bosons $B_{m}$ 
	and they generate a Heisenberg algebra \cite{U}. 
For a partition $\lambda$, we define $S_{\lambda}$ as a $\mathbb{Q}$-linear combination of products of elements $B_{-m}$. 
	(See \S 3.1 for the precise definition.) 
Quantum group $U_{q}(\hat{\mathfrak{sl}}_{n})$ also acts on $\boldsymbol{F}_{q}[\boldsymbol{s}]$ as level $q^{\ell}$. 
These actions commute on $\boldsymbol{F}_{q}[\boldsymbol{s}]$. 

The canonical bases 
$\{ G^{+}(\boldsymbol{\lambda} ; \boldsymbol{s}) \, | \, \boldsymbol{\lambda} \in \Pi^{\ell} \}$ and 
$\{ G^{-}(\boldsymbol{\lambda} ; \boldsymbol{s}) \, | \, \boldsymbol{\lambda} \in \Pi^{\ell} \}$ are 
	bases of the Fock space $\boldsymbol{F}_{q}[\boldsymbol{s}]$ 
	that are invariant under a certain involution $\overline{\phantom{a}}$ \cite{U}.

A partition $\lambda = (\lambda_{1} , \lambda_{2} , \cdots )$ is {\it $n$-restricted} if 
	$0 \le \lambda_{i} - \lambda_{i+1} < n$ for all $i=1,2,\cdots$. 
Each partition $\lambda$ can be uniquely written as $\lambda = \widetilde{\lambda} + n \check{\lambda}$, 
	where $\widetilde{\lambda}, \check{\lambda} \in \Pi$ and $\widetilde{\lambda}$ is $n$-restricted. 
In the case of $\ell =1$, Leclerc and Thibon introduced linear operators 
	$S_{\lambda}$ as a $\mathbb{Q}$-linear combination of products of elements $B_{-m}$ and showed the following theorem
	(See \S 3.1 for the precise definitions) :

\smallskip 

{ \bf Theorem A. }
(Leclerc-Thibon\cite{LT}) \\
\it Suppose that $\ell =1$. 
Let $\lambda = \widetilde{\lambda} + n \check{\lambda}$. Then, 
\begin{equation*}
G^{-}(\lambda) = S_{\check{\lambda}} G^{-} (\widetilde{\lambda}) .
\end{equation*}

\smallskip

\noindent 
\rm 
Our main result is a higher level version of Theorem A. 
For a non-negative integer $M$, 
	a basis vector $| \boldsymbol{\lambda} ; \boldsymbol{s} \rangle$ is called {\it $M$-dominant} if 
$s_{i} - s_{i+1} \geq M + | \boldsymbol{\lambda} | $ 
for all $i=1,2,\cdots,\ell-1$. 
For $\boldsymbol{\lambda} = (\lambda^{(1)} , \lambda^{(2)} , \cdots , \lambda^{(\ell)}) \in \Pi^{\ell}$, 
define $\widetilde{\Bf{\lambda}}$ and $\check{\Bf{\lambda}}$ as 
$\widetilde{\Bf{\lambda}} = (\widetilde{\lambda^{(1)}} , \widetilde{\lambda^{(2)}} , \cdots , \widetilde{\lambda^{(\ell)}}) 
	\quad , \quad 
	\check{\Bf{\lambda}} = (\check{\lambda^{(1)}} , \check{\lambda^{(2)}} , \cdots , \check{\lambda^{(\ell)}}) $. 

\smallskip

{\bf Theorem B. (Theorem \ref{maincor})} \\
\it
If $| \Bf{\lambda} ; \Bf{s} \rangle$ is $0$-dominant, then 
\begin{equation*}
	G^{-}(\Bf{\lambda} ; \Bf{s}) 
	= S_{ \check{\Bf{\lambda}} } G^{-} (\widetilde{ \Bf{\lambda} } ; \Bf{s}) , 
\end{equation*}
where $S_{ \check{\Bf{\lambda}} }$ is a linear operator on $\Bf{F}_q[\Bf{s}]$ defined in Definition \ref{def.4.11}. 

\smallskip
\rm

Now we explain a expected connections of canonical bases $G^{\pm}(\Bf{\lambda} ; \Bf{s})$ with representation theory. 
Define matrices  
$\Delta^{+}(q) = (\Delta^{+}_{\boldsymbol{\lambda},\boldsymbol{\mu}}(q))
_{\boldsymbol{\lambda},\boldsymbol{\mu}}$ and 
$\Delta^{-}(q) = (\Delta^{-}_{\boldsymbol{\lambda},\boldsymbol{\mu}}(q))
_{\boldsymbol{\lambda},\boldsymbol{\mu}}$ by 
$G^{+}(\boldsymbol{\lambda} ; \boldsymbol{s}) = \sum_{\boldsymbol{\mu}} 
\Delta^{+}_{\boldsymbol{\lambda},\boldsymbol{\mu}}(q) \, | \, \boldsymbol{\mu} ; \boldsymbol{s} \rangle $, 
$G^{-}(\boldsymbol{\lambda} ; \boldsymbol{s}) = \sum_{\boldsymbol{\mu}} 
\Delta^{-}_{\boldsymbol{\lambda},\boldsymbol{\mu}}(q) \, | \, \boldsymbol{\mu} ; \boldsymbol{s} \rangle $.
We call $\Delta^{+}_{\boldsymbol{\lambda},\boldsymbol{\mu}}(q)$ and $\Delta^{-}_{\boldsymbol{\lambda},\boldsymbol{\mu}}(q)$ 
{\it $q$-decomposition numbers}. 
These $q$-decomposition matrices play an important role in representation theory. 
However it is not known that there is a explicit combinatorial formula 
	which expresses $q$-decomposition matrices as an one dimensional summation.  

In the case of $\ell = 1$, 
Varagnolo-Vasserot \cite{VV1} proved that $\Delta^{+}(1)$ coincides with the decomposition matrix of $v$-Schur algebra. 
Ariki defined a $q$-analogue of decomposition numbers of $v$-Schur algebra by using Khovanov-Lauda's grading, 
	and proved that it coincides with the $q$-decomposition numbers \cite{A}. 

We recall the tensor product theorem by Lusztig \cite{Lus}. 
Let $\zeta \in \mathbb{C}$ be such that $\zeta^2$ is a primitive $n$-th root of unity. 
Let $\mathrm{Fr}$ denote the Frobenius map from the quantum enveloping algebra $U_{\zeta}(\mathfrak{gl}_r)$ 
	to the classical enveloping algebra $U(\mathfrak{gl}_r)$. 
Given a $U(\mathfrak{gl}_r)$-module $M$, one can define a $U_{\zeta}(\mathfrak{gl}_r)$-modulte $M^{\mathrm{Fr}}$ 
	by composing the action of $U(\mathfrak{gl}_r)$ with Fr. 

\smallskip 

{ \bf Theorem }
(Steinberg-Lusztig \cite{Lus}) \\
\it Suppose that $\ell =1$. 
Let $\lambda = \widetilde{\lambda} + n \check{\lambda}$. Then, 
\begin{equation*}
L(\lambda) = L(\widetilde{\lambda}) \otimes W({\check{\lambda}})^{\mathrm{Fr}} ,
\end{equation*}
where $L(\lambda)$ (resp. $W(\mu)$) is the simple (resp. classical Weyl) module with highest weight $\lambda$ ($\mu$). 

\smallskip

\rm 
Since the simple module $L(\lambda)$ corresponds to the canonical basis $G^{-}(\lambda)$ in the case of $\ell = 1$, 
	Theorem A is a Fock space version of Steinberg-Lusztig's tensor product theorem. (see \cite{LT} for details.)

For $\ell \ge 2$, Yvonne \cite{Y} conjectured that the matrix $\Delta^{+}(q)$ coincides with 
the $q$-analogue of the decomposition matrices of cyclotomic Schur algebras at a primitive $n$-th root of unity 
under a suitable condition on a multi charge and a proof of the conjecture is presented by 
	Stroppel-Webster \cite[Theorem D]{StWe}. 

Let $\mathcal{O}_{\boldsymbol{s}}( \ell , 1, m)$ be the category $\mathcal{O}$ of rational Cherednik algebra of 
	$( \mathbb{Z} / \ell \mathbb{Z} ) \wr \mathfrak{S}_{m}$ associated with multi charge $\boldsymbol{s}$ \cite{GGOR}. 
Rouquier \cite[Theorem 6.8, \S 6.5]{R2} conjectured that, for an arbitrary multi charge, 
	the multiplicities of simple modules in standard modules in $\mathcal{O}_{\boldsymbol{s}}( \ell , 1, m)$  
	are equal to the corresponding coefficients $\Delta^{+}_{\boldsymbol{\lambda},\boldsymbol{\mu}}(q)$, 
	where $m = | \boldsymbol{\lambda} | = | \boldsymbol{\mu} |$. 
Shan showed that $\oplus_{m \ge 0} \mathcal{O}_{\boldsymbol{s}}( \ell , 1, m)$ 
	categorify $\boldsymbol{F}_{1}[\boldsymbol{s}]$ \cite{S}.  
More generally, it is expected that, together with a suitable grading, $\oplus_{m \ge 0} \mathcal{O}_{\boldsymbol{s}}( \ell , 1, m)$ 
	should categorify $\boldsymbol{F}_{q}[\boldsymbol{s}]$. 
For the detail of correspondence between the charges of $\mathcal{O}_{\boldsymbol{s}}( \ell , 1, m)$ and 
	the charges of Fock spaces, see \cite{R2}. 

%

\smallskip 

This paper is organized as follows. 
In Section 2, we review the $q$-deformed Fock spaces of higher levels and their canonical bases. 
In Section 3, we review the $q$-analogue of tensor product theorem by Leclerc-Thibon. 
In Section 4, we state the main results and prove them except Proposition \ref{mainconj}.  
In Section 5, we prove Proposition \ref{mainconj}.

\subsection*{Acknowledgments}

I am deeply grateful to Hyohe Miyachi, Soichi Okada, Toshiaki Shoji and Kentaro Wada for their advice.

%
%

\subsection*{Notations}

For a positive integer $N$, a \it{partition} \rm{of} $N$ is 
a non-increasing sequence of non-negative integers summing to $N$. 
We write $|\lambda| = N$ if $\lambda$ is a partition of $N$.
The \it{length} $l(\lambda)$ \rm{of} $\lambda$ is the number of non-zero components of $\lambda$.
And we use the same notation $\lambda$ to represent the Young diagram corresponding to $\lambda$.
For an $\ell$-tuple $\boldsymbol{\lambda} = (\lambda^{(1)} , \lambda^{(2)} , \cdots , \lambda^{(\ell)})$ of Young diagrams, 
	we put $|\boldsymbol{\lambda}| = |\lambda^{(1)}| + |\lambda^{(2)}| + \cdots + |\lambda^{(\ell)}|$.

%
%

\section{The $q$-deformed Fock spaces of higher levels}

\subsection{$q$-wedge products and straightening rules}

Let $n$, $\ell$, $s$ be integers such that $n \ge 2$ and $\ell \ge 1$. 
Let $r \in \mathbb{Z}_{\ge 0}$. 
For $\Bf{k} \in \mathbb{Z}^{r}$, a {\it finite $q$-wedge of length $r$} is 
$$
u_{\Bf{k}} = u_{k_{1}} \wedge u_{k_2} \wedge \cdots \wedge u_{k_r} .
$$
Finite $q$-wedges satisfy certain commutation relations, so-called {\it straightening rules. }
Note that the straightening rules depend on $n$ and $\ell$. 
\cite[Proposition 3.16]{U}

\begin{example}
$(i)$ For every $k_{1} \in \mathbb{Z}$, $u_{k_{1}} \wedge u_{k_{1}} = - u_{k_{1}} \wedge u_{k_{1}}$. 
Therefore $u_{k_{1}} \wedge u_{k_{1}} = 0$.

$(ii)$ Let $n=2$, $\ell =2$, $k_{1}=-2$, and $k_{2}=4$. 
Then 
$$ u_{-2} \wedge u_{4} = q \, u_{4} \wedge u_{-2} + (q^{2} - 1) \, u_{2} \wedge u_{0}. $$

$(iii)$ Let $n=2$, $\ell =2$, $k_{1}=-1$, $k_{2}=-2$ and $k_{3}=4$. 
Then, 
\begin{align*}
u_{-1} \wedge u_{-2} \wedge u_{4} = u_{-1} \wedge ( u_{-2} \wedge u_{4} ) 
&= u_{-1} \wedge \Big( q \, u_{4} \wedge u_{-2} + (q^{2} - 1) \, u_{2} \wedge u_{0} \Big) \\
&= q \, u_{-1} \wedge u_{4} \wedge u_{-2} + (q^{2} - 1) \, u_{-1} \wedge u_{2} \wedge u_{0} \\
&= - u_{4} \wedge u_{-1} \wedge u_{-2} - (q-q^{-1}) u_{3} \wedge u_{0} \wedge u_{-2} 
	- (q-q^{-1}) u_{2} \wedge u_{0} \wedge u_{-1} . 
\end{align*}
\label{ex3}
\end{example}

We define $P(s)$ and $P^{++}(s)$ as follows;
\begin{align}
P(s) &= \{ \boldsymbol{k} = ( k_{1} , k_{2} , \cdots ) \in \mathbb{Z}^{\infty}  \,\, | \,\, 
	k_{r} = s-r+1 \,\, \text{ for any sufficiently large } r \,\, \}  \\
P^{++}(s) &= \{ \boldsymbol{k} = ( k_{1} , k_{2} , \cdots ) \in P(s)  \,\, | \,\, k_{1} > k_{2} > \cdots \,\, \}  .
\end{align}
Let $\Lambda^{s}$ be the $\mathbb{Q}(q)$ vector space spanned by the $q$-wedge products 
\begin{equation}
u_{\boldsymbol{k}} = u_{k_{1}} \wedge u_{k_{2}} \wedge \cdots \,\,\,\, , \,\,\,\, (\boldsymbol{k} \in P(s))
\end{equation}
subject to straightening rules. 
We regard a finite $q$-wedge product $u_{k_{1}} \wedge u_{k_2} \wedge \cdots \wedge u_{k_r}$ as 
the infinite $q$-wedge product 
\begin{equation}
u_{k_{1}} \wedge u_{k_2} \wedge \cdots \wedge u_{k_r} \wedge u_{s-r} \wedge u_{s-r-1} \wedge u_{s-r-2} \wedge \cdots  .
\end{equation}

By applying the straightening rules, every $q$-wedge product $u_{\boldsymbol{k}}$ is 
expressed as a linear combination of so-called 
\it{ordered $q$-wedge products}, \rm{namely} $q$-wedge products $u_{\boldsymbol{k}}$ with $\boldsymbol{k} \in P^{++}(s)$.
The ordered $q$-wedge products $\{ u_{\boldsymbol{k}}  \,\, | \,\, \boldsymbol{k} \in P^{++}(s) \}$ 
form a basis of $\Lambda^{s}$ called \it{the standard basis.} \rm{}

\subsection{Abacus}

It is convenient to use the abacus notation for studying various properties in straightening rules. 

Fix an integer $N \ge 2$, and form an infinite abacus with $N$ runners labeled $1,2,\cdots N$ from left to right. 
The positions on the $i$-th runner are labeled by the integers having residue $i$ modulo $N$. 

\begin{equation*}
\begin{array}{ccccc}
\vdots & \vdots & \vdots & \vdots & \vdots \\
-N+1 & -N+2 & \cdots & -1 & 0 \\
1 & 2 & \cdots & N-1 & N \\
N+1 & N+2 & \cdots & 2N-1 & 2N \\
\vdots & \vdots & \vdots & \vdots & \vdots
\end{array}
\end{equation*}

Each $\boldsymbol{k} \in P^{++}(s)$ (or the corresponding $q$-wedge product $u_{\boldsymbol{k}}$) \
can be represented by a bead-configuration on the abacus with $n \ell$ runners 
and beads put on the positions $k_{1},k_{2},\cdots$.
We call this configuration {\it the abacus presentation} of $u_{\boldsymbol{k}}$. 

\begin{example}
If $n=2$, $\ell =3$, $s=0$, and $\boldsymbol{k} = (6,3,2,1,-2,-4,-5,-7,-8,-9,\cdots)$, 
then the abacus presentation of $u_{\boldsymbol{k}}$ is 

$$\begin{array}{cc|cc|cccc}
d=1 & & d=2 & & d=3 & & & \\ 
\vdots & \vdots & \vdots & \vdots & \vdots & \vdots & & \\
\centercircle{-17} & \centercircle{-16} & \centercircle{-15} & \centercircle{-14} & \centercircle{-13} & \centercircle{-12} 
		& \hspace{1em} & \cdots m=3  \\ 
\centercircle{-11} & \centercircle{-10} & \centercircle{-9} & \centercircle{-8} & \centercircle{-7} & -6 & \hspace{1em} & \cdots m=2  \\ 
\centercircle{-5} & \centercircle{-4} & -3 & \centercircle{-2} & -1 & 0 & \hspace{1em} & \cdots m=1  \\ 
\centercircle{1} & \centercircle{2} & \centercircle{3} & 4 & 5 & \centercircle{6} & \hspace{1em} & \cdots m=0  \\
\vdots & \vdots & \vdots & \vdots & \vdots & \vdots & & \\
c=1 & c= 2 & c=1 & c= 2 & c=1 & c= 2 & &
\end{array}$$
\end{example}

We use another labeling of runners and positions. 
Given an integer $k$, let $c,d$ and $m$ be the unique integers satisfying 
\begin{equation}
k=c+n(d-1)-n \ell m \quad , \quad 1 \le c \le n \quad \text{ and } \quad 1 \le d \le \ell. 
\label{k->cdm}
\end{equation}
Then, in the abacus presentation, the position $k$ is on the $c + n(d-1)$-th runner (see the previous example). 
Relabeling the position $k$ by $c - n m$, we have $\ell$ abaci with $n$ runners.

\begin{example}
In the previous example, relabeling the position $k$ by $c - n m$, we have 

$$\begin{array}{cc|cc|cccc}
d=1 & & d=2 & & d=3 & & & \\ 
\vdots & \vdots & \vdots & \vdots & \vdots & \vdots & & \\
\centercircle{-5} & \centercircle{-4} & \centercircle{-5} & \centercircle{-4} & \centercircle{-5} & \centercircle{-4} 
		& \hspace{1em} & \cdots m=3  \\ 
\centercircle{-3} & \centercircle{-2} & \centercircle{-3} & \centercircle{-2} & \centercircle{-3} & -2 & \hspace{1em} & \cdots m=2  \\ 
\centercircle{-1} & \centercircle{0} & -1 & \centercircle{0} & -1 & 0 & \hspace{1em} & \cdots m=1  \\ 
\centercircle{1} & \centercircle{2} & \centercircle{1} & 2 & 1 & \centercircle{2} & \hspace{1em} & \cdots m=0  \\
\vdots & \vdots & \vdots & \vdots & \vdots & \vdots & & \\
c=1 & c= 2 & c=1 & c= 2 & c=1 & c= 2 & &
\end{array}$$

\label{exam.2.5}
\end{example}

We assign to each of $\ell$ abacus presentations with $n$ runners a $q$-wedge product of level $1$. 
In fact, straightening rules in each ``sector'' are the same as those of level $1$ 
	by identifying the abacus in the sector with that of level $1$. 
(see also \cite{U} and \cite{I})

We introduce some notation.

\begin{definition}
For an integer $k$, let $c,d$ and $m$ be the unique integers satisfying (\ref{k->cdm}), and write 

\begin{equation}
u_{k} = u_{c - n m}^{(d)}.
\end{equation}

Also we write $u_{c_{1} - n m_{1}}^{(d_{1})} > u_{c_{2} - n m_{2}}^{(d_{2})}$ 
if $k_{1} > k_{2}$, where $k_{i} = c_{i} + n (d_{i} - 1) - n \ell m_{i}$, $(i=1,2)$. 

A finite $q$-wedge 
$$ v=u_{k_1}^{(d_1)} \wedge u_{k_2}^{(d_2)} \wedge \cdots \wedge u_{k_r}^{(d_r)} $$ 
is {\it simple} if $d_1=d_2= \cdots = d_r$.
\label{notation}
\end{definition}

We regard $u_{c - n m}^{(d)}$ as $u_{c - n m}$ in the case of $\ell = 1$. 

\begin{example}
If $n=2$, $\ell = 3$, then we have 
$$ u_{-10} \wedge u_{1} = -q^{-1} \, u_{1} \wedge u_{-10} +(q^{-2} -1) \, u_{-4} \wedge u_{-5}, $$
that is, 
$$ u_{-2}^{(1)} \wedge u_{1}^{(1)} = -q^{-1} \, u_{1}^{(1)} \wedge u_{-2}^{(1)} +(q^{-2} -1) \, u_{0}^{(1)} \wedge u_{-1}^{(1)}. $$

On the other hand, in the case of $n=2, \ell =1$, 
$$ u_{-2} \wedge u_{1} = -q^{-1} \, u_{1} \wedge u_{-2} +(q^{-2} -1) \, u_{0} \wedge u_{-1}.$$
\end{example}

\subsection{$\ell$-tuples of Young diagrams}

Another indexation of the ordered $q$-wedge products is given by the set of pairs 
$( \boldsymbol{\lambda} , \boldsymbol{s} )$ 
of $\ell$-tuples of Young diagrams $\boldsymbol{\lambda} = ( \lambda^{(1)} , \cdots ,\lambda^{(\ell)})$
and integer sequence $\boldsymbol{s} = (s_{1} , \cdots , s_{\ell})$ summing up to $s$.
Let $\boldsymbol{k} = (k_{1} , k_{2} , \cdots ) \in P^{++}(s)$, and write 
$$k_{r} = c_{r} + n(d_{r} - 1) - n \ell m_{r} \quad , 
	\quad 1 \le c_{r} \le n \quad , \quad 1 \le d_{r} \le \ell \quad , \quad m_{r} \in \mathbb{Z} \quad . $$
For $d \in \{ 1 ,2 , \cdots , \ell \}$, let $k_{1}^{(d)}, k_{2}^{(d)}, \cdots$ be integers such that 

$$\beta^{(d)} = \{ c_{r} - n m_{r} \,\, | \,\, d_{r} = d \} = \{ k_{1}^{(d)}, k_{2}^{(d)}, \cdots \} \quad \text{ and } \quad 
	k_{1}^{(d)} > k_{2}^{(d)} > \cdots $$ 
Then we associate to the sequence $(k_{1}^{(d)} , k_{2}^{(d)} , \cdots )$ an integer $s_{d}$ 
	and a partition $\lambda^{(d)}$ by 
$$ k_{r}^{(d)} = s_{d} -r+1 \quad \text{ for sufficiently large } r 
	\quad \text{ and } \quad  
	\lambda^{(d)}_{r} = k^{(d)}_{r} - s_{d} + r -1 \quad  \text { for } r \ge 1. $$
In this correspondence, we also write 

\begin{equation}
u_{\boldsymbol{k}} = | \boldsymbol{\lambda} ; \boldsymbol{s} \rangle \quad (\boldsymbol{k} \in P^{++}(s)).
\end{equation}

\begin{example}
If $n=2$, $\ell =3$, $s=0$, and $\boldsymbol{k} = (6,3,2,1,-2,-4,-5,-7,-8,-9,\cdots)$, then

\begin{align*}
k_{1} &= 6 = 2 + 2(3-1) -6 \cdot 0 \,\,\,\,,\,\,\,\,
k_{2} = 3 = 1 + 2(2-1) -6 \cdot 0 \,\,\,\,, \\
k_{3} &= 2 = 2 + 2(1-1) -6 \cdot 0 \,\,\,\,\,\, ,  \cdots \text{ and so on.} 
\end{align*}
Hence, 
\begin{align*}
\beta^{(1)} = \{ 2,1,0,-1,-2, \cdots \} \,\,\,\,\, , \,\,\,\,\,
\beta^{(2)} = \{ 1,0,-2,-3,-4, \cdots \}  \,\,\,\,\,\, , \,\,\,\,\,
\beta^{(3)} = \{ 2,-3,-4,-5, \cdots \}  \,\,\,\, .
\end{align*}
Thus, $\boldsymbol{s} = (2,0,-2)$ and $\boldsymbol{\lambda} = (\emptyset , (1,1) , (4))$.

Note that we can read off $\boldsymbol{s} = (2,0,-2)$ and $\boldsymbol{\lambda} = (\emptyset , (1,1) , (4))$ 
from the abacus presentation. 
$($see Example \ref{exam.2.5}$)$
\end{example}

\subsection{The $q$-deformed Fock spaces of higher levels}

\begin{definition}
For $\boldsymbol{s} \in \mathbb{Z}^{\ell}$, we define the $q$-deformed Fock space 
$\boldsymbol{F}_{q}[\boldsymbol{s}]$ of level $\ell$ to be the subspace of $\Lambda^{s}$ spanned by 
$| \boldsymbol{\lambda} ; \boldsymbol{s} \rangle$ $(\boldsymbol{\lambda} \in \Pi^{\ell}) \colon $  

\begin{equation}
\boldsymbol{F}_{q}[\boldsymbol{s}] = 
\bigoplus_{\boldsymbol{\lambda} \in \Pi^{\ell}} \mathbb{Q}(q) \, | \boldsymbol{\lambda} ; \boldsymbol{s} \rangle.
\end{equation}
We call $\boldsymbol{s}$ a {\it multi charge}. 
\end{definition}

\subsection{The action of bosons}

The Fock space $\Bf{F}_{q}[\Bf{s}]$ is endowed with the action of bosons $B_{m}$ given by 
\begin{equation}
	B_{m} (u_{\Bf{k}}) 
	= \sum_{r \ge 1} u_{k_{1}} \wedge u_{k_{2}} \wedge \cdots \wedge u_{k_{r-1}} \wedge u_{k_{r}-n \ell m} 
	\wedge u_{k_{r+1}} \wedge \cdots \quad , \quad ( m \in \mathbb{Z}^{*} ), 
\end{equation}
where $\mathbb{Z}^{*}$ denotes the set of nonzero integers. 

{ \bf Remark}. $B_{m}$'s generate a Heisenberg algebra \cite[Proposition 4.5]{U}: 
\begin{equation*}
	[B_{m} , B_{m'}] = \delta_{m,-m'} \, m \, \frac{1 - q^{-2mn}}{1-q^{-2m}} \cdot \frac{1-q^{2m \ell}}{1-q^{2m}} , 
	\quad (m \in \mathbb{Z}_{>0} \text{ and }  m' \in \mathbb{Z}^{*}).
\end{equation*}

\subsection{The bar involution}

\begin{definition}
The bar involution $\overline{\phantom{xy}}$ of $\Lambda^{s}$ is the $\mathbb{Q}$-vector space automorphism such that 
$\overline{q} = q^{-1}$ and 

\begin{align}
\overline{ u_{\boldsymbol{k}} } 
&= \overline{ u_{k_{1}} \wedge \cdots \wedge u_{k_{r}} } \wedge u_{k_{r+1}} \wedge \cdots 
=(-q)^{\kappa (d_{1}, \cdots ,d_{r})} q^{-\kappa (c_{1},\cdots,c_{r})}  
(u_{k_{r}} \wedge \cdots \wedge u_{k_{1}} )\wedge u_{k_{r+1}} \wedge \cdots ,
\label{barinvolution}
\end{align}
where $c_{i}$, $d_{i}$ are defined by $k_{i}$ as in (\ref{k->cdm}), $r$ is an integer satisfying $k_{r} = s-r+1$. 
And $\kappa (a_{1}, \cdots ,a_{r})$ is defined by  
$$ \kappa (a_{1}, \cdots ,a_{r}) = \# \{ (i,j) \, | \, i<j \,,\, a_{i}=a_{j} \} .$$
\label{barinv}
\end{definition}

{\bf Remarks}. 
\begin{enumerate}
	\item The involution is well defined. i.e. it doesn't depend on the choice of $r$ \cite{U}. 
	\item The involution comes from the bar involution of affine Hecke algebra $\hat{H_{r}}$. 
		(see \cite{U} for more detail.)
	\item The involution preserves the $q$-deformed Fock space $\boldsymbol{F}_{q}[\boldsymbol{s}]$ of multi charge $\Bf{s}$. 
\end{enumerate}

The following proposition shows that the action of $B_{m}$ commutes with the bar involution. 

\begin{proposition}[\cite{U}]
For $| \Bf{\lambda} ; \Bf{s} \rangle \in \Bf{F}_{q}[\Bf{s}] $ and $m \in \mathbb{Z}_{>0}$, we have 
\begin{equation}
	\overline{ B_{-m} | \Bf{\lambda} ; \Bf{s} \rangle } = B_{-m} \overline{ | \Bf{\lambda} ; \Bf{s} \rangle }. 
\end{equation}
\label{barcommute}
\end{proposition}

\subsection{The dominance order}

We define a partial ordering 
$ | \boldsymbol{\lambda} ; \boldsymbol{s} \rangle \ge | \boldsymbol{\mu} ; \boldsymbol{s} \rangle $. 
%

\begin{definition}
Let $ | \boldsymbol{\lambda} ; \boldsymbol{s} \rangle  = u_{k_{1}} \wedge u_{k_{2}} \wedge \cdots \,\,$ and \,\,
$ | \boldsymbol{\mu} ; \boldsymbol{s} \rangle = u_{g_{1}} \wedge u_{g_{2}} \wedge \cdots  $. 
We define $ | \boldsymbol{\lambda} ; \boldsymbol{s} \rangle \ge | \boldsymbol{\mu} ; \boldsymbol{s} \rangle $ 
if $| \boldsymbol{\lambda} | = | \boldsymbol{\mu} |$ and 

\begin{equation}
\sum_{j=1}^{r} k_{j} \ge \sum_{j=1}^{r} g_{j} \,\,\,\,\,\, (\text{for all } \,\,\, r=1,2,3,\cdots ) \quad .
\end{equation}
\label{order}
\end{definition}


\begin{example}
Let $n= \ell =2$, $\boldsymbol{s} = (1,-1)$, $\boldsymbol{\lambda} = ((1,1) , \emptyset)$, and
$\boldsymbol{\mu} = (\emptyset , (2))$. 
Then, $| \boldsymbol{\lambda} ; \boldsymbol{s} \rangle = u_{2} \wedge u_{1} \wedge u_{-1} \wedge u_{-3} \wedge \cdots$ and 
$| \boldsymbol{\mu} ; \boldsymbol{s} \rangle = u_{3} \wedge u_{1} \wedge u_{-2} \wedge u_{-3} \wedge \cdots$.
Thus, $| \boldsymbol{\mu} ; \boldsymbol{s} \rangle$ is greater than 
	$| \boldsymbol{\lambda} ; \boldsymbol{s} \rangle$. 
\label{example1}
\end{example}

We define a matrix $(a_{\boldsymbol{\lambda},\boldsymbol{\mu}}(q))_{\boldsymbol{\lambda},\boldsymbol{\mu}}$ by 

\begin{equation}
\overline{ |\boldsymbol{\lambda} ; \boldsymbol{s} \rangle } = 
\sum_{\boldsymbol{\mu}} a_{\boldsymbol{\lambda},\boldsymbol{\mu}}(q) \,
|\boldsymbol{\mu} ; \boldsymbol{s} \rangle .
\end{equation}
Then the matrix $(a_{\boldsymbol{\lambda},\boldsymbol{\mu}}(q))_{\boldsymbol{\lambda},\boldsymbol{\mu}}$
is unitriangular with respect to $\ge$, that is 
\begin{equation}
\begin{cases}
\mathrm{(a)} & 
\text{ if } \,\, a_{\boldsymbol{\lambda},\boldsymbol{\mu}}(q) \not= 0 \,\, \text{, then } \,\,
|\boldsymbol{\lambda} ; \boldsymbol{s} \rangle \ge |\boldsymbol{\mu} ; \boldsymbol{s} \rangle , \\
\mathrm{(b)} & 
a_{\boldsymbol{\lambda},\boldsymbol{\lambda}}(q) = 1 .
\label{unitriangularity}
\end{cases}
\end{equation}
(see the identity (\ref{eq.22}) for the detail.)

Thus, by the standard argument, the unitriangularity implies the following theorem.

\begin{theorem}\cite[Theorem 3.25]{U}
There exist unique bases 
$\{ G^{+}(\boldsymbol{\lambda} ; \boldsymbol{s}) \, | \, \boldsymbol{\lambda} \in \Pi^{\ell}  \}$ and 
$\{ G^{-}(\boldsymbol{\lambda} ; \boldsymbol{s}) \, | \, \boldsymbol{\lambda} \in \Pi^{\ell}  \}$ 
of $\boldsymbol{F}_{q}[\boldsymbol{s}]$ such that 

\begin{align*}
\mathrm{(i) } \hspace{5em}
\overline{G^{+}(\boldsymbol{\lambda} ; \boldsymbol{s})} = G^{+}(\boldsymbol{\lambda} ; \boldsymbol{s}) 
\hspace{2em} , \hspace{3em} & 
\overline{G^{-}(\boldsymbol{\lambda} ; \boldsymbol{s})} = G^{-}(\boldsymbol{\lambda} ; \boldsymbol{s}) \\
\mathrm{(ii) } \hspace{1em}
G^{+}(\boldsymbol{\lambda} ; \boldsymbol{s}) \equiv | \, \boldsymbol{\lambda} ; \boldsymbol{s} \rangle
\,\,\,\, \mathrm{mod} \,\, q \, \mathcal{L}^{+} 
\hspace{2em} , \hspace{3em} & 
G^{-}(\boldsymbol{\lambda} ; \boldsymbol{s}) \equiv | \, \boldsymbol{\lambda} ; \boldsymbol{s} \rangle
\,\,\,\, \mathrm{mod} \,\, q^{-1} \, \mathcal{L}^{-}   \\
\text{where} \hspace{5em}
\mathcal{L}^{+} = \bigoplus_{\boldsymbol{\lambda} \in \Pi^{\ell}} 
\mathbb{Q}[q] \, | \boldsymbol{\lambda} ; \boldsymbol{s} \rangle
\hspace{2em} , \hspace{3em} & 
\mathcal{L}^{-} = \bigoplus_{\boldsymbol{\lambda} \in \Pi^{\ell}} 
\mathbb{Q}[q^{-1}] \, | \boldsymbol{\lambda} ; \boldsymbol{s} \rangle .
\end{align*}
\end{theorem}

\begin{definition}
Define matrices 
$\Delta^{+}(q) = (\Delta^{+}_{\boldsymbol{\lambda},\boldsymbol{\mu}}(q))
_{\boldsymbol{\lambda},\boldsymbol{\mu}}$ and 
$\Delta^{-}(q) = (\Delta^{-}_{\boldsymbol{\lambda},\boldsymbol{\mu}}(q))
_{\boldsymbol{\lambda},\boldsymbol{\mu}}$ by 

\begin{align}
G^{+}(\boldsymbol{\lambda} ; \boldsymbol{s}) = \sum_{\boldsymbol{\mu}} 
\Delta^{+}_{\boldsymbol{\lambda},\boldsymbol{\mu}}(q) \, | \, \boldsymbol{\mu} ; \boldsymbol{s} \rangle 
\hspace{2em} , \hspace{3em} 
G^{-}(\boldsymbol{\lambda} ; \boldsymbol{s}) = \sum_{\boldsymbol{\mu}} 
\Delta^{-}_{\boldsymbol{\lambda},\boldsymbol{\mu}}(q) \, | \, \boldsymbol{\mu} ; \boldsymbol{s} \rangle .
\end{align}
\end{definition}

The entries $\Delta^{\pm}_{\boldsymbol{\lambda},\boldsymbol{\mu}}(q)$ are called {\it $q$-decomposition numbers}. 
Note that $q$-decomposition numbers $\Delta^{\pm}(q)$ depend on $n$, $\ell$ and $\boldsymbol{s}$.
The matrices $\Delta^{+}(q)$ and $\Delta^{-}(q)$ are also unitriangular with respect to $\ge$.

It is known \cite[Theorem 3.26]{U} that the entries of $\Delta^{-}(q)$ are 
	Kazhdan-Lusztig polynomials of parabolic submodules of affine Hecke algebras of type $A$, 
	and that they are polynomials in $q$ with non-negative integer coefficients.

%
%

\section{A $q$-analogue of the tensor product theorem of level one}

In this section we review the $q$-analogue of tensor product theorem in the case of $\ell = 1$ \cite{LT}. 

\subsection{$V_{\lambda}$ and $S_{\lambda}$}

Let $p_{m}$, $h_{m}$ and $s_{\lambda}$ be the power sum symmetric function of degree $m$, 
the complete symmetric function of degree $m$ and Schur function, respectively. 
There are some well-known relationship among them. 

\begin{equation}
	h_{m} = \sum_{|\lambda |= m} \frac{1}{z_{\lambda}} p_{\lambda} \quad , \quad 
	s_{\lambda} = \sum_{\mu} K_{\mu , \lambda}^{(-1)} h_{\mu} \quad , 
\end{equation}
where $K_{\mu , \lambda}^{(-1)}$ is the inverse Kostka number 
	and for a partition $\mu = (1^{\alpha_{1}} , 2^{\alpha_{2}} , \cdots  )$, 
	we define $z_{\mu} = \Pi_{i \ge 1} i^{\alpha_{i}} \alpha_{i}!$. 
For a partition $\lambda = (\lambda_{1} , \lambda_{2} , \cdots )$, we define $B_{-\lambda}$ by 
\begin{equation}
	B_{-\lambda} = B_{-\lambda_{1}} B_{-\lambda_{2}} \cdots .
\end{equation}

\begin{definition}
For $m \in \mathbb{Z}_{>0}$ and $\lambda \in \Pi$, 
we define operators $V_{m}$ and $S_{\lambda}$ on $\Bf{F}_{q}[\Bf{s}]$ as 
\begin{equation}
	V_{m} = \sum_{|\lambda |= m} \frac{1}{z_{\lambda}} B_{-\lambda} \quad , \quad 
	S_{\lambda} = \sum_{\mu} K_{\mu , \lambda}^{(-1)} V_{\mu} \quad , 
\end{equation}
where $V_{\mu} = V_{\mu_{1}} V_{\mu_{2}} \cdots $ for $\mu = (\mu_{1} , \mu_{2} , \cdots ) \in \Pi$.
\label{p-h-s}
\end{definition}

That is, we regard $B_{-m}$ (resp. $V_{m}$, $S_{\lambda}$) 
	as the power sum (resp. the complete symmetric function, Schur function). 
By Proposition \ref{barcommute}, $V_{m}$ and $S_{\lambda}$ also commute the bar involution. 

The action of $V_m$ is combinatorially described as follows. 
An {\it $n$-ribbon} is a connected strip of $n$-cells which does not contain a $2 \times 2$ square; 
more precisely, an $n$-ribbon is a sequence of $n$ cells $R=\{ (a_1,b_1) , (a_2,b_2) , \cdots , (a_n,b_n) \}$ such that 
	$(a_{i+1} , b_{i+1})$ is either $(a_{i}+1 , b_{i})$ or $(a_{i} , b_{i}-1)$, for $i=1,2,\cdots,n$. 
The {\it head} of $R$ is the cell $\mathrm{head}(R) = (a_1,b_1)$ and 
	$\mathrm{spin}_{n}(R) = \# \{ 1 \le i<n \mid a_{i+1}=a_{i}+1 \}$ is the {\it $n$-spin} of $R$.

For partitions $\lambda$ and $\mu$, we write $\lambda \overset{m \colon n}{\rightsquigarrow} \mu$ 
	if $\lambda \subset \mu$ and the skew diagram $\mu \backslash \lambda$ is a disjoint union of $m$ $n$-ribbons such that 
	the head of each ribbon is either in the first row of $\mu$ or is of the form $(i,j)$ where $(i-1,j) \in \lambda$.
Lascoux, Leclerc and Thibon call $\mu / \lambda$ an {\it $n$-ribbon tableau} of weight $(m)$ and 
	they note that there is a unique way of writing $\mu \backslash \lambda$ as a disjoint union of ribbons. 
Finally, if $\lambda \overset{m \colon n}{\rightsquigarrow} \mu$ then $\mathrm{spin}_{n}(\mu / \lambda)$, 
	the $n$-spin of $\mu / \lambda$, is the sum of the $n$-spins of the ribbons in $\mu \backslash \lambda$. 

\begin{theorem}[\cite{LT} Theorem.6.7]
Let $m \in \mathbb{Z}_{\ge 0}$ and $\lambda \in \Pi$.  
If $\ell =1$, then 
$$
V_{m} \, | \lambda \rangle 
	= \sum_{ \lambda \overset{m \colon n}{\rightsquigarrow} \mu } (-q^{-1})^{\mathrm{spin}_{n}(\mu / \lambda)} \, | \mu \rangle .
$$
\label{LT.th6.7}
\end{theorem}

\begin{example}
Let $n=3$, $\ell =1$, $m=2$ and $\lambda=(2)$. Then, 
$$
V_{2} \, | (2) ; s \rangle =  | (8) ; s \rangle -q^{-1} \, | (5,2,1) ; s \rangle + q^{-2} | (4,3,2) ; s \rangle 
	+ q^{-2} \, | (5,1^3) ; s \rangle + q^{-4} \, | (2^4) ; s \rangle - q^{-3} \, | (3^{2},1^{2}) ; s \rangle .
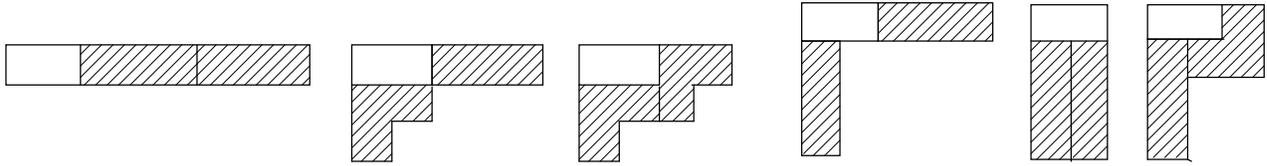
$$

\unitlength 0.1in
\begin{picture}( 38.0000,  6.1000)(  0.0000, -8.0000)
%
\special{pn 8}%
\special{pa 0 190}%
\special{pa 1590 190}%
\special{pa 1590 400}%
\special{pa 0 400}%
\special{pa 0 190}%
\special{fp}%
%
\special{pn 8}%
\special{pa 390 190}%
\special{pa 390 400}%
\special{fp}%
%
\special{pn 8}%
\special{pa 1000 190}%
\special{pa 1000 400}%
\special{fp}%
%
\special{pn 4}%
\special{pa 860 190}%
\special{pa 650 400}%
\special{fp}%
\special{pa 800 190}%
\special{pa 590 400}%
\special{fp}%
\special{pa 740 190}%
\special{pa 530 400}%
\special{fp}%
\special{pa 680 190}%
\special{pa 470 400}%
\special{fp}%
\special{pa 620 190}%
\special{pa 410 400}%
\special{fp}%
\special{pa 560 190}%
\special{pa 390 360}%
\special{fp}%
\special{pa 500 190}%
\special{pa 390 300}%
\special{fp}%
\special{pa 440 190}%
\special{pa 390 240}%
\special{fp}%
\special{pa 920 190}%
\special{pa 710 400}%
\special{fp}%
\special{pa 980 190}%
\special{pa 770 400}%
\special{fp}%
\special{pa 1000 230}%
\special{pa 830 400}%
\special{fp}%
\special{pa 1000 290}%
\special{pa 890 400}%
\special{fp}%
\special{pa 1000 350}%
\special{pa 950 400}%
\special{fp}%
%
\special{pn 4}%
\special{pa 1340 190}%
\special{pa 1130 400}%
\special{fp}%
\special{pa 1280 190}%
\special{pa 1070 400}%
\special{fp}%
\special{pa 1220 190}%
\special{pa 1010 400}%
\special{fp}%
\special{pa 1160 190}%
\special{pa 1000 350}%
\special{fp}%
\special{pa 1100 190}%
\special{pa 1000 290}%
\special{fp}%
\special{pa 1040 190}%
\special{pa 1000 230}%
\special{fp}%
\special{pa 1400 190}%
\special{pa 1190 400}%
\special{fp}%
\special{pa 1460 190}%
\special{pa 1250 400}%
\special{fp}%
\special{pa 1520 190}%
\special{pa 1310 400}%
\special{fp}%
\special{pa 1580 190}%
\special{pa 1370 400}%
\special{fp}%
\special{pa 1590 240}%
\special{pa 1430 400}%
\special{fp}%
\special{pa 1590 300}%
\special{pa 1490 400}%
\special{fp}%
\special{pa 1590 360}%
\special{pa 1550 400}%
\special{fp}%
%
\special{pn 8}%
\special{pa 1810 190}%
\special{pa 2230 190}%
\special{pa 2230 400}%
\special{pa 1810 400}%
\special{pa 1810 190}%
\special{fp}%
%
\special{pn 8}%
\special{pa 2230 190}%
\special{pa 2810 190}%
\special{pa 2810 400}%
\special{pa 2230 400}%
\special{pa 2230 190}%
\special{fp}%
%
\special{pn 8}%
\special{pa 1810 400}%
\special{pa 1810 800}%
\special{fp}%
\special{pa 1810 800}%
\special{pa 2020 800}%
\special{fp}%
%
\special{pn 8}%
\special{pa 2020 590}%
\special{pa 2020 800}%
\special{fp}%
\special{pa 2020 590}%
\special{pa 2230 590}%
\special{fp}%
\special{pa 2230 590}%
\special{pa 2230 410}%
\special{fp}%
%
\special{pn 4}%
\special{pa 2080 400}%
\special{pa 1810 670}%
\special{fp}%
\special{pa 2140 400}%
\special{pa 1810 730}%
\special{fp}%
\special{pa 2200 400}%
\special{pa 1810 790}%
\special{fp}%
\special{pa 2020 640}%
\special{pa 1860 800}%
\special{fp}%
\special{pa 2020 700}%
\special{pa 1920 800}%
\special{fp}%
\special{pa 2020 760}%
\special{pa 1980 800}%
\special{fp}%
\special{pa 2230 430}%
\special{pa 2070 590}%
\special{fp}%
\special{pa 2230 490}%
\special{pa 2130 590}%
\special{fp}%
\special{pa 2230 550}%
\special{pa 2190 590}%
\special{fp}%
\special{pa 2020 400}%
\special{pa 1810 610}%
\special{fp}%
\special{pa 1960 400}%
\special{pa 1810 550}%
\special{fp}%
\special{pa 1900 400}%
\special{pa 1810 490}%
\special{fp}%
\special{pa 1840 400}%
\special{pa 1810 430}%
\special{fp}%
%
\special{pn 4}%
\special{pa 2590 190}%
\special{pa 2380 400}%
\special{fp}%
\special{pa 2530 190}%
\special{pa 2320 400}%
\special{fp}%
\special{pa 2470 190}%
\special{pa 2260 400}%
\special{fp}%
\special{pa 2410 190}%
\special{pa 2230 370}%
\special{fp}%
\special{pa 2350 190}%
\special{pa 2230 310}%
\special{fp}%
\special{pa 2290 190}%
\special{pa 2230 250}%
\special{fp}%
\special{pa 2650 190}%
\special{pa 2440 400}%
\special{fp}%
\special{pa 2710 190}%
\special{pa 2500 400}%
\special{fp}%
\special{pa 2770 190}%
\special{pa 2560 400}%
\special{fp}%
\special{pa 2810 210}%
\special{pa 2620 400}%
\special{fp}%
\special{pa 2810 270}%
\special{pa 2680 400}%
\special{fp}%
\special{pa 2810 330}%
\special{pa 2740 400}%
\special{fp}%
%
\special{pn 8}%
\special{pa 3000 190}%
\special{pa 3420 190}%
\special{pa 3420 400}%
\special{pa 3000 400}%
\special{pa 3000 190}%
\special{fp}%
%
\special{pn 8}%
\special{pa 3000 400}%
\special{pa 3000 800}%
\special{fp}%
\special{pa 3000 800}%
\special{pa 3210 800}%
\special{fp}%
%
\special{pn 8}%
\special{pa 3210 590}%
\special{pa 3210 800}%
\special{fp}%
\special{pa 3210 590}%
\special{pa 3420 590}%
\special{fp}%
\special{pa 3420 590}%
\special{pa 3420 410}%
\special{fp}%
%
\special{pn 4}%
\special{pa 3270 400}%
\special{pa 3000 670}%
\special{fp}%
\special{pa 3330 400}%
\special{pa 3000 730}%
\special{fp}%
\special{pa 3390 400}%
\special{pa 3000 790}%
\special{fp}%
\special{pa 3210 640}%
\special{pa 3050 800}%
\special{fp}%
\special{pa 3210 700}%
\special{pa 3110 800}%
\special{fp}%
\special{pa 3210 760}%
\special{pa 3170 800}%
\special{fp}%
\special{pa 3420 430}%
\special{pa 3260 590}%
\special{fp}%
\special{pa 3420 490}%
\special{pa 3320 590}%
\special{fp}%
\special{pa 3420 550}%
\special{pa 3380 590}%
\special{fp}%
\special{pa 3210 400}%
\special{pa 3000 610}%
\special{fp}%
\special{pa 3150 400}%
\special{pa 3000 550}%
\special{fp}%
\special{pa 3090 400}%
\special{pa 3000 490}%
\special{fp}%
\special{pa 3030 400}%
\special{pa 3000 430}%
\special{fp}%
%
\special{pn 8}%
\special{pa 3420 190}%
\special{pa 3800 190}%
\special{fp}%
\special{pa 3800 190}%
\special{pa 3800 400}%
\special{fp}%
\special{pa 3800 400}%
\special{pa 3590 400}%
\special{fp}%
\special{pa 3600 400}%
\special{pa 3600 590}%
\special{fp}%
\special{pa 3600 590}%
\special{pa 3420 590}%
\special{fp}%
%
\special{pn 4}%
\special{pa 3710 190}%
\special{pa 3420 480}%
\special{fp}%
\special{pa 3770 190}%
\special{pa 3420 540}%
\special{fp}%
\special{pa 3600 420}%
\special{pa 3430 590}%
\special{fp}%
\special{pa 3600 480}%
\special{pa 3490 590}%
\special{fp}%
\special{pa 3600 540}%
\special{pa 3550 590}%
\special{fp}%
\special{pa 3800 220}%
\special{pa 3620 400}%
\special{fp}%
\special{pa 3800 280}%
\special{pa 3680 400}%
\special{fp}%
\special{pa 3800 340}%
\special{pa 3740 400}%
\special{fp}%
\special{pa 3650 190}%
\special{pa 3420 420}%
\special{fp}%
\special{pa 3590 190}%
\special{pa 3420 360}%
\special{fp}%
\special{pa 3530 190}%
\special{pa 3420 300}%
\special{fp}%
\special{pa 3470 190}%
\special{pa 3420 240}%
\special{fp}%
\end{picture}%
 \quad \quad 
\unitlength 0.1in
\begin{picture}( 24.2000,  8.3000)(  2.0000,-10.3000)
%
\special{pn 8}%
\special{pa 200 200}%
\special{pa 600 200}%
\special{pa 600 400}%
\special{pa 200 400}%
\special{pa 200 200}%
\special{fp}%
%
\special{pn 8}%
\special{pa 600 200}%
\special{pa 1200 200}%
\special{pa 1200 400}%
\special{pa 600 400}%
\special{pa 600 200}%
\special{fp}%
%
\special{pn 8}%
\special{pa 200 400}%
\special{pa 400 400}%
\special{pa 400 1000}%
\special{pa 200 1000}%
\special{pa 200 400}%
\special{fp}%
%
\special{pn 4}%
\special{pa 1060 200}%
\special{pa 860 400}%
\special{fp}%
\special{pa 1000 200}%
\special{pa 800 400}%
\special{fp}%
\special{pa 940 200}%
\special{pa 740 400}%
\special{fp}%
\special{pa 880 200}%
\special{pa 680 400}%
\special{fp}%
\special{pa 820 200}%
\special{pa 620 400}%
\special{fp}%
\special{pa 760 200}%
\special{pa 600 360}%
\special{fp}%
\special{pa 700 200}%
\special{pa 600 300}%
\special{fp}%
\special{pa 640 200}%
\special{pa 600 240}%
\special{fp}%
\special{pa 1120 200}%
\special{pa 920 400}%
\special{fp}%
\special{pa 1180 200}%
\special{pa 980 400}%
\special{fp}%
\special{pa 1200 240}%
\special{pa 1040 400}%
\special{fp}%
\special{pa 1200 300}%
\special{pa 1100 400}%
\special{fp}%
\special{pa 1200 360}%
\special{pa 1160 400}%
\special{fp}%
%
\special{pn 4}%
\special{pa 400 620}%
\special{pa 200 820}%
\special{fp}%
\special{pa 400 680}%
\special{pa 200 880}%
\special{fp}%
\special{pa 400 740}%
\special{pa 200 940}%
\special{fp}%
\special{pa 400 800}%
\special{pa 210 990}%
\special{fp}%
\special{pa 400 860}%
\special{pa 260 1000}%
\special{fp}%
\special{pa 400 920}%
\special{pa 320 1000}%
\special{fp}%
\special{pa 400 560}%
\special{pa 200 760}%
\special{fp}%
\special{pa 400 500}%
\special{pa 200 700}%
\special{fp}%
\special{pa 400 440}%
\special{pa 200 640}%
\special{fp}%
\special{pa 380 400}%
\special{pa 200 580}%
\special{fp}%
\special{pa 320 400}%
\special{pa 200 520}%
\special{fp}%
\special{pa 260 400}%
\special{pa 200 460}%
\special{fp}%
%
\special{pn 8}%
\special{pa 1400 210}%
\special{pa 1800 210}%
\special{pa 1800 1020}%
\special{pa 1400 1020}%
\special{pa 1400 210}%
\special{fp}%
%
\special{pn 8}%
\special{pa 1400 400}%
\special{pa 1800 400}%
\special{fp}%
\special{pa 1610 400}%
\special{pa 1610 1030}%
\special{fp}%
%
\special{pn 4}%
\special{pa 1610 720}%
\special{pa 1400 930}%
\special{fp}%
\special{pa 1610 780}%
\special{pa 1400 990}%
\special{fp}%
\special{pa 1610 840}%
\special{pa 1430 1020}%
\special{fp}%
\special{pa 1610 900}%
\special{pa 1490 1020}%
\special{fp}%
\special{pa 1610 960}%
\special{pa 1550 1020}%
\special{fp}%
\special{pa 1610 660}%
\special{pa 1400 870}%
\special{fp}%
\special{pa 1610 600}%
\special{pa 1400 810}%
\special{fp}%
\special{pa 1610 540}%
\special{pa 1400 750}%
\special{fp}%
\special{pa 1610 480}%
\special{pa 1400 690}%
\special{fp}%
\special{pa 1610 420}%
\special{pa 1400 630}%
\special{fp}%
\special{pa 1570 400}%
\special{pa 1400 570}%
\special{fp}%
\special{pa 1510 400}%
\special{pa 1400 510}%
\special{fp}%
\special{pa 1450 400}%
\special{pa 1400 450}%
\special{fp}%
%
\special{pn 4}%
\special{pa 1800 710}%
\special{pa 1610 900}%
\special{fp}%
\special{pa 1800 770}%
\special{pa 1610 960}%
\special{fp}%
\special{pa 1800 830}%
\special{pa 1620 1010}%
\special{fp}%
\special{pa 1800 890}%
\special{pa 1670 1020}%
\special{fp}%
\special{pa 1800 950}%
\special{pa 1730 1020}%
\special{fp}%
\special{pa 1800 650}%
\special{pa 1610 840}%
\special{fp}%
\special{pa 1800 590}%
\special{pa 1610 780}%
\special{fp}%
\special{pa 1800 530}%
\special{pa 1610 720}%
\special{fp}%
\special{pa 1800 470}%
\special{pa 1610 660}%
\special{fp}%
\special{pa 1800 410}%
\special{pa 1610 600}%
\special{fp}%
\special{pa 1750 400}%
\special{pa 1610 540}%
\special{fp}%
\special{pa 1690 400}%
\special{pa 1610 480}%
\special{fp}%
%
\special{pn 8}%
\special{pa 2010 390}%
\special{pa 2410 390}%
\special{fp}%
\special{pa 2220 390}%
\special{pa 2220 1020}%
\special{fp}%
%
\special{pn 4}%
\special{pa 2220 710}%
\special{pa 2010 920}%
\special{fp}%
\special{pa 2220 770}%
\special{pa 2010 980}%
\special{fp}%
\special{pa 2220 830}%
\special{pa 2040 1010}%
\special{fp}%
\special{pa 2220 890}%
\special{pa 2100 1010}%
\special{fp}%
\special{pa 2220 950}%
\special{pa 2160 1010}%
\special{fp}%
\special{pa 2220 650}%
\special{pa 2010 860}%
\special{fp}%
\special{pa 2220 590}%
\special{pa 2010 800}%
\special{fp}%
\special{pa 2220 530}%
\special{pa 2010 740}%
\special{fp}%
\special{pa 2220 470}%
\special{pa 2010 680}%
\special{fp}%
\special{pa 2220 410}%
\special{pa 2010 620}%
\special{fp}%
\special{pa 2180 390}%
\special{pa 2010 560}%
\special{fp}%
\special{pa 2120 390}%
\special{pa 2010 500}%
\special{fp}%
\special{pa 2060 390}%
\special{pa 2010 440}%
\special{fp}%
%
\special{pn 8}%
\special{pa 2010 210}%
\special{pa 2400 210}%
\special{pa 2400 390}%
\special{pa 2010 390}%
\special{pa 2010 210}%
\special{fp}%
%
\special{pn 8}%
\special{pa 2010 390}%
\special{pa 2010 1020}%
\special{fp}%
\special{pa 2010 1020}%
\special{pa 2220 1020}%
\special{fp}%
\special{pa 2220 1020}%
\special{pa 2240 1030}%
\special{fp}%
%
\special{pn 8}%
\special{pa 2400 210}%
\special{pa 2620 210}%
\special{fp}%
%
\special{pn 8}%
\special{pa 2620 210}%
\special{pa 2620 590}%
\special{fp}%
%
\special{pn 8}%
\special{pa 2220 590}%
\special{pa 2620 590}%
\special{fp}%
%
\special{pn 4}%
\special{pa 2620 410}%
\special{pa 2440 590}%
\special{fp}%
\special{pa 2620 350}%
\special{pa 2380 590}%
\special{fp}%
\special{pa 2620 290}%
\special{pa 2320 590}%
\special{fp}%
\special{pa 2620 230}%
\special{pa 2260 590}%
\special{fp}%
\special{pa 2400 390}%
\special{pa 2220 570}%
\special{fp}%
\special{pa 2340 390}%
\special{pa 2220 510}%
\special{fp}%
\special{pa 2280 390}%
\special{pa 2220 450}%
\special{fp}%
\special{pa 2580 210}%
\special{pa 2400 390}%
\special{fp}%
\special{pa 2520 210}%
\special{pa 2400 330}%
\special{fp}%
\special{pa 2460 210}%
\special{pa 2400 270}%
\special{fp}%
\special{pa 2620 470}%
\special{pa 2500 590}%
\special{fp}%
\special{pa 2620 530}%
\special{pa 2560 590}%
\special{fp}%
\end{picture}%

\end{example}


\subsection{A $q$-analogue of the tensor product theorem in the case of $\ell=1$}

\begin{definition}
A partition $\lambda = (\lambda_{1} , \lambda_{2} , \cdots )$ is {\it $n$-restricted} is 
\begin{equation}
	0 \le \lambda_{i} - \lambda_{i+1} < n \quad \text{for all $i=1,2,\cdots$}.
\end{equation}
\end{definition}

\begin{definition}
For $\lambda \in \Pi$, we define $\widetilde{\lambda}, \check{\lambda} \in \Pi$ as 
$\widetilde{\lambda}$ is $n$-restricted and $\lambda = \widetilde{\lambda} + n \check{\lambda}$. 
\label{tilde-check}
\end{definition}

\begin{example}
Let $n=3$ and $\lambda = (9,5,4,4)$. 
Then $(9,5,4,4) = 3 \cdot (2,1,1,1) + (3,2,1,1)$. 
Thus, $\widetilde{\lambda} = (3,2,1,1), \check{\lambda} = (2,1,1,1)$. 
\end{example}

\unitlength 0.1in
\begin{picture}( 40.0000,  8.1000)(  2.0000,-10.1000)
%
\special{pn 8}%
\special{pa 200 200}%
\special{pa 1990 200}%
\special{fp}%
\special{pa 210 200}%
\special{pa 210 1010}%
\special{fp}%
\special{pa 210 990}%
\special{pa 1000 990}%
\special{fp}%
%
\special{pn 8}%
\special{pa 1010 600}%
\special{pa 1010 990}%
\special{fp}%
\special{pa 1010 610}%
\special{pa 1210 610}%
\special{fp}%
\special{pa 1210 410}%
\special{pa 1210 610}%
\special{fp}%
\special{pa 1210 400}%
\special{pa 1990 400}%
\special{fp}%
\special{pa 1990 200}%
\special{pa 1990 410}%
\special{fp}%
%
\special{pn 8}%
\special{pa 800 200}%
\special{pa 800 1000}%
\special{fp}%
%
\special{pn 8}%
\special{pa 1210 200}%
\special{pa 1800 200}%
\special{pa 1800 400}%
\special{pa 1210 400}%
\special{pa 1210 200}%
\special{fp}%
%
\special{pn 4}%
\special{pa 800 400}%
\special{pa 220 980}%
\special{fp}%
\special{pa 800 460}%
\special{pa 270 990}%
\special{fp}%
\special{pa 800 520}%
\special{pa 330 990}%
\special{fp}%
\special{pa 800 580}%
\special{pa 390 990}%
\special{fp}%
\special{pa 800 640}%
\special{pa 450 990}%
\special{fp}%
\special{pa 800 700}%
\special{pa 510 990}%
\special{fp}%
\special{pa 800 760}%
\special{pa 570 990}%
\special{fp}%
\special{pa 800 820}%
\special{pa 630 990}%
\special{fp}%
\special{pa 800 880}%
\special{pa 690 990}%
\special{fp}%
\special{pa 800 940}%
\special{pa 750 990}%
\special{fp}%
\special{pa 800 340}%
\special{pa 210 930}%
\special{fp}%
\special{pa 800 280}%
\special{pa 210 870}%
\special{fp}%
\special{pa 800 220}%
\special{pa 210 810}%
\special{fp}%
\special{pa 760 200}%
\special{pa 210 750}%
\special{fp}%
\special{pa 700 200}%
\special{pa 210 690}%
\special{fp}%
\special{pa 640 200}%
\special{pa 210 630}%
\special{fp}%
\special{pa 580 200}%
\special{pa 210 570}%
\special{fp}%
\special{pa 520 200}%
\special{pa 210 510}%
\special{fp}%
\special{pa 460 200}%
\special{pa 210 450}%
\special{fp}%
\special{pa 400 200}%
\special{pa 210 390}%
\special{fp}%
\special{pa 340 200}%
\special{pa 210 330}%
\special{fp}%
\special{pa 280 200}%
\special{pa 210 270}%
\special{fp}%
%
\special{pn 4}%
\special{pa 1720 200}%
\special{pa 1520 400}%
\special{fp}%
\special{pa 1660 200}%
\special{pa 1460 400}%
\special{fp}%
\special{pa 1600 200}%
\special{pa 1400 400}%
\special{fp}%
\special{pa 1540 200}%
\special{pa 1340 400}%
\special{fp}%
\special{pa 1480 200}%
\special{pa 1280 400}%
\special{fp}%
\special{pa 1420 200}%
\special{pa 1220 400}%
\special{fp}%
\special{pa 1360 200}%
\special{pa 1210 350}%
\special{fp}%
\special{pa 1300 200}%
\special{pa 1210 290}%
\special{fp}%
\special{pa 1240 200}%
\special{pa 1210 230}%
\special{fp}%
\special{pa 1780 200}%
\special{pa 1580 400}%
\special{fp}%
\special{pa 1800 240}%
\special{pa 1640 400}%
\special{fp}%
\special{pa 1800 300}%
\special{pa 1700 400}%
\special{fp}%
\special{pa 1800 360}%
\special{pa 1760 400}%
\special{fp}%
\put(21.3000,-6.0000){\makebox(0,0)[lb]{$=$}}%
\put(24.5000,-6.0000){\makebox(0,0)[lb]{$3$}}%
%
\special{pn 8}%
\special{pa 2810 200}%
\special{pa 3010 200}%
\special{pa 3010 1010}%
\special{pa 2810 1010}%
\special{pa 2810 200}%
\special{fp}%
%
\special{pn 8}%
\special{pa 3010 200}%
\special{pa 3200 200}%
\special{pa 3200 400}%
\special{pa 3010 400}%
\special{pa 3010 200}%
\special{fp}%
%
\special{pn 4}%
\special{pa 3010 470}%
\special{pa 2810 670}%
\special{fp}%
\special{pa 3010 530}%
\special{pa 2810 730}%
\special{fp}%
\special{pa 3010 590}%
\special{pa 2810 790}%
\special{fp}%
\special{pa 3010 650}%
\special{pa 2810 850}%
\special{fp}%
\special{pa 3010 710}%
\special{pa 2810 910}%
\special{fp}%
\special{pa 3010 770}%
\special{pa 2810 970}%
\special{fp}%
\special{pa 3010 830}%
\special{pa 2830 1010}%
\special{fp}%
\special{pa 3010 890}%
\special{pa 2890 1010}%
\special{fp}%
\special{pa 3010 950}%
\special{pa 2950 1010}%
\special{fp}%
\special{pa 3010 410}%
\special{pa 2810 610}%
\special{fp}%
\special{pa 3010 350}%
\special{pa 2810 550}%
\special{fp}%
\special{pa 3010 290}%
\special{pa 2810 490}%
\special{fp}%
\special{pa 3010 230}%
\special{pa 2810 430}%
\special{fp}%
\special{pa 2980 200}%
\special{pa 2810 370}%
\special{fp}%
\special{pa 2920 200}%
\special{pa 2810 310}%
\special{fp}%
\special{pa 2860 200}%
\special{pa 2810 250}%
\special{fp}%
%
\special{pn 4}%
\special{pa 3200 280}%
\special{pa 3080 400}%
\special{fp}%
\special{pa 3200 220}%
\special{pa 3020 400}%
\special{fp}%
\special{pa 3160 200}%
\special{pa 3010 350}%
\special{fp}%
\special{pa 3100 200}%
\special{pa 3010 290}%
\special{fp}%
\special{pa 3040 200}%
\special{pa 3010 230}%
\special{fp}%
\special{pa 3200 340}%
\special{pa 3140 400}%
\special{fp}%
\put(33.0000,-6.0000){\makebox(0,0)[lb]{$+$}}%
%
\special{pn 8}%
\special{pa 3610 200}%
\special{pa 3610 1010}%
\special{fp}%
\special{pa 3610 1000}%
\special{pa 3800 1000}%
\special{fp}%
\special{pa 3800 1000}%
\special{pa 3800 600}%
\special{fp}%
\special{pa 3800 600}%
\special{pa 4000 600}%
\special{fp}%
\special{pa 4000 600}%
\special{pa 4000 200}%
\special{fp}%
\special{pa 4000 200}%
\special{pa 3600 200}%
\special{fp}%
%
\special{pn 8}%
\special{pa 4000 200}%
\special{pa 4200 200}%
\special{pa 4200 390}%
\special{pa 4000 390}%
\special{pa 4000 200}%
\special{fp}%
\end{picture}%

\begin{theorem}[Leclerc-Thibon\cite{LT}, Theorem 6.9]

Let $\lambda = \widetilde{\lambda} + n \check{\lambda}$. Then 
\begin{equation}
G^{-}(\lambda) = S_{\check{\lambda}} G^{-} (\widetilde{\lambda}) .
\end{equation}
\label{LTformula}
\end{theorem}

{\bf Remark.} Theorem \ref{LTformula} is a formal analogue of Lusztig's tensor product theorem. 
(see \cite[\S 3.2]{LT}.)

\begin{example}
In this example, we write $| \lambda ; s \rangle$ simply $\lambda$. 

(i) Let $n=2$, $\lambda = (4) = \emptyset + n (2)$. 
Then, 

\begin{align*}
	G^{-}((4)) 
	&= S_{(2)} \, G^{-}(\emptyset)  
	=V_{(2)} \, \emptyset 
	= (4) - q^{-1} (3 , 1) + q^{-2} (2 , 2) .
\end{align*}

(ii) Let $n=2$ and $\lambda = (2 , 2) = \emptyset + n (1 , 1)$. Then, 
\begin{align*}
	G^{-}((2 , 2)) 
	&= S_{(1 , 1)} G^{-}(\emptyset) 
	=(V_{(1,1)} - V_{2}) \emptyset  
	= (2 , 2) - q^{-1} (2 , 1 , 1) + q^{-2} (1^{4}) .
\end{align*}
\end{example}

%
%

\section{A $q$-analogue of the tensor product theorem of higher levels}

\subsection{$M$-dominancy}

\begin{definition}
For $M \in \mathbb{Z}_{\ge 0}$, $| \boldsymbol{\lambda} ; \boldsymbol{s} \rangle$ is {\it $M$-dominant} if 
\begin{equation}
	s_{i} - s_{i+1} \geq M + | \boldsymbol{\lambda} | 
\end{equation}
for all $i=1,2,\cdots,\ell-1$. 
\end{definition}

\begin{definition}
For $1 \leq i \leq l$, we define linear operators $B_{-m}'[i]$ by 

\begin{equation}
B_{-m}'[i] \, | (\lambda^{(1)} , \cdots , \lambda^{(i)} , \cdots,  \lambda^{(l)}) ; \boldsymbol{s} \rangle =
| (\lambda^{(1)} , \cdots , B_{-m} \lambda^{(i)} , \cdots,  \lambda^{(l)}) ; \boldsymbol{s} \rangle  ,
\end{equation}
where the right hand side is understood as 
$$
| (\lambda^{(1)} , \cdots , B_{-m} \lambda^{(i)} , \cdots,  \lambda^{(l)}) ; \boldsymbol{s} \rangle 
= \sum_{\mu} c_{\mu} | (\lambda^{(1)} , \cdots , \mu , \cdots,  \lambda^{(l)}) ; \boldsymbol{s} \rangle
$$
if $B_{-m} | \lambda^{(i)} \rangle = \sum_{\mu} c_{\mu} | \mu \rangle$ in the $q$-deformed Fock space of level one. 
\end{definition}

\begin{proposition}[\cite{U}, Proposition 5.3(i)] 

Let $m \in \mathbb{Z}_{> 0}$. 
Suppose that $| \boldsymbol{\lambda} ; \boldsymbol{s} \rangle$ is $nm$-dominant. 
Then, 

\begin{equation}
B_{-m} \, | \boldsymbol{\lambda} ; \boldsymbol{s} \rangle = 
\sum_{i=1}^{l} q^{(i-1)m} B_{-m}'[i] \, | \boldsymbol{\lambda} ; \boldsymbol{s} \rangle . 
\end{equation}
\label{uglov5}
\end{proposition}

\subsection{A $q$-analogue of the tensor product theorem of higher levels}

\begin{definition}
For $1 \le j \le \ell$, we define $B_{-m}[j]$ as follows. 
\begin{align}
	B_{-m}[\ell] = B_{-m}'[\ell] \quad , \quad 
	B_{-m}[j] = B_{-m}'[j] - q^{-m} B_{-m}'[j+1] \quad , \quad (1 \le j \le \ell-1).
\end{align}
For $1 \le j \le \ell$, 
\begin{align}
	B_{-m}[j,\ell] = \sum_{i = j}^{\ell} q^{(i-j)m} B_{-m}'[i] .
\end{align}
\label{defB}
\end{definition}

\begin{lemma}
If $u = | \boldsymbol{\lambda} ; \boldsymbol{s} \rangle$ is $nm$-dominant, then 
\begin{align}
	B_{-m} \, u &= \sum_{j=1}^{\ell} \frac{q^{jm} - q^{-jm}}{q^{m} - q^{-m}} B_{-m}[j] \, u.
\label{hoshi}
\end{align}

\end{lemma}

\begin{proof}
\begin{align*}
	\sum_{j=1}^{\ell} \frac{q^{jm} - q^{-jm}}{q^{m} - q^{-m}} B_{-m}[j] 
	&= \sum_{j=1}^{\ell} \frac{q^{jm} - q^{-jm}}{q^{m} - q^{-m}} B_{-m}'[j]  
		- q^{-m} \sum_{j=1}^{\ell-1} \frac{q^{jm} - q^{-jm}}{q^{m} - q^{-m}} B_{-m}'[j+1] \\
	&= B_{-m}'[1]  
		+ \sum_{j=2}^{\ell} \left( \frac{q^{jm} - q^{-jm}}{q^{m} - q^{-m}} - q^{-m} \frac{q^{(j-1)m} - q^{-(j-1)m}}{q^{m} - q^{-m}} \right) 
		B_{-m}'[j]  \\
	&= B_{-m}'[1] + \sum_{j=2}^{\ell} q^{(j-1)m} B_{-m}'[j] .
\end{align*}
Hence, the assertion follows from Proposition \ref{uglov5}. 
\end{proof}

The next proposition is the key proposition in this paper and will be proved in section 5. 

\begin{proposition}
Let $m \in \mathbb{Z}_{>0}$ and $1 \le i \le \ell$. 
If $u = | \boldsymbol{\lambda} ; \boldsymbol{s} \rangle$ is $nm$-dominant, then \\
\begin{enumerate}
\item[(i)]
$\dis{\overline{B_{-m}[j] \, u} = B_{-m}[j] \, \overline{u}. }$
\item[(ii)]
$\dis{\overline{B_{-m}[j,\ell] \, u} = B_{-m}[j,\ell] \, \overline{u}. }$
\end{enumerate}
\label{mainconj}
\end{proposition}

We define $V_{k}'[i]$, $V_{k}[i]$, $S_{\lambda}'[i]$, $S_{\lambda}[i]$ in the same fashion in Definition \ref{p-h-s}. 

\begin{definition} 
For $1 \le i \le \ell$, 
we define $V_{k}'[i]$, $V_{k}[i]$, $S_{\lambda}'[i]$, $S_{\lambda}[i]$ as follows. 
\begin{align*}
	V_{m}'[i] = \sum_{| \lambda | = m} \frac{1}{z_{\lambda}} B_{- \lambda}'[i]  \quad , \quad 
	V_{m}[i] = \sum_{| \lambda | = m} \frac{1}{z_{\lambda}} B_{- \lambda}[i]  \quad , \\
	S_{\lambda}'[i] = \sum_{\mu} K_{\mu , \lambda}^{(-1)} V_{\mu}'[i] \quad , \quad 
	S_{\lambda}[i] = \sum_{\mu} K_{\mu , \lambda}^{(-1)} V_{\mu}[i] \quad .
\end{align*}
\end{definition}

{\bf Remarks}. 
(i). From the definition of $B_{-m}'[i]$, 
the operator $V_{m}'[i]$ (resp. $S_{\mu}'[i]$) acts on the $i$-th component of $| \Bf{\lambda} ; \Bf{s} \rangle$ 
in the same way as $V_{m}$ (resp. $S_{\mu}$) in the case of $\ell =1$. 
If $n=2$, for example, 
\begin{align*}
	V_{m}'[1] \, | (\lambda^{(1)},\lambda^{(2)}) ; \Bf{s} \rangle = | (V_{m} \lambda^{(1)},\lambda^{(2)}) ; \Bf{s} \rangle \quad , \quad 
	S_{\mu}'[2] \, | (\lambda^{(1)},\lambda^{(2)}) ; \Bf{s} \rangle = | (\lambda^{(1)} , S_{\mu} \lambda^{(2)}) ; \Bf{s} \rangle .
\end{align*}
In particular, $V_{m}'[i]$ has a combinatorial expression gives in Theorem \ref{LT.th6.7}.

(ii). Since $B_{-m}[\ell] = B_{-m}'[\ell]$, we have $ V_{m}[\ell] = V_{m}'[\ell]$ and $S_{\lambda}[\ell] = S_{\lambda}'[\ell]$. 

\vspace{1em}

The next lemma gives the formula that expresses $S_{\lambda}[i]$ in terms of $S_{\mu}'[i]$ and $S_{\nu}'[i+1]$. 
Therefore we can compute $S_{\lambda}[i] \, u$ from the calculation in the case of $\ell =1$ 
	if $u$ is $n |\lambda| $-dominant. 

\begin{lemma}
For $1 \le j < \ell$, 
\begin{equation}
	S_{\lambda}[j]  = 
	\sum_{\mu, \nu} (-q^{-1})^{| \nu |} \mathrm{LR}^{\lambda}_{\mu \, \nu} \, S_{\mu}'[j] \, S_{\transpose{\nu}}'[j+1] \quad , 
\label{lemmaformula}
\end{equation}
where $\mathrm{LR}^{\lambda}_{\mu \, \nu}$ is the Littlewood-Richardson coefficient and  
$\transpose{\nu}$ means the transpose of the Young diagram $\nu$.
\label{keylemma}
\end{lemma}

\begin{proof}
Let $X=(X_{1},X_{2},\cdots), Y=(Y_{1},Y_{2},\cdots)$ be two families of variables. 
Let $\Lambda(X|Y) = \Lambda(X) \otimes \Lambda(Y)$ be the ring consisting of all symmetric function in both $X$ and $Y$. 
Let $\mathcal{B} = \mathbb{Q}[ B_{-m}'[i] \, | \, m \in \mathbb{Z}_{\ge 0} \,,\, i=j,j+1]$ 
	be the ring consisting of all the $\mathbb{Q}$-linear combinations of products of elements $B_{-m}'[j]$ and $B_{-m}'[j+1]$. 
We define the ring homomorphism $\iota \colon \Lambda(X|Y) \rightarrow \mathcal{B}$ by 
$\iota (p_m(X)) = B_{-m}'[j]$ and $\iota (p_m(Y)) = B_{-m}'[j+1]$. 

We define an automorphism $\omega_{-Y}$ on $\Lambda(X|Y)$ as follows. 
$$ 
\omega_{-Y} (p_{m}(X)) = p_{m}(X) \quad , \quad  \omega_{-Y} (p_{m}(Y)) = -p_{m}(Y) \qquad (m=1,2,\cdots).
$$
If we denote $q^{-1}Y$ by $(q^{-1}Y_{1},q^{-1}Y_{2},\cdots)$, then 
$$
\omega_{-Y} (p_m(X,q^{-1}Y)) 
= \omega_{-Y}(p_m(X)+p_m(q^{-1}Y)) 
= \omega_{-Y}(p_m(X))+q^{-m} \omega_{-Y}(p_m(Y)) 
= p_m(X) -q^{-m} p_m(Y)
$$
Hence we have $B_{-m}[j]=\iota ( \omega_{-Y}(p_m(X,q^{-1}Y)) )$, 
	and it follows from Definition \ref{p-h-s} that 
\begin{equation}
S_{\lambda}[j] = \iota ( \omega_{-Y}(s_{\lambda}(X,q^{-1}Y)) ) ,
\label{omega-Y}
\end{equation}
where $s_{\lambda}(X,q^{-1}Y)$ is the supersymmetric Schur function with variable 
	$(X,q^{-1}Y) = (X_{1} , X_{2} , \cdots , q^{-1}Y_{1} , q^{-1}Y_{2} , \cdots) $. 

Now we compute $\omega_{-Y}(s_{\lambda}(X,q^{-1}Y))$. 
Note that  
\begin{align*}
	\omega_{-Y}(h_{k}(Y)) &= \omega_{-Y} \biggl( \sum_{| \lambda | = k} \frac{1}{z_{\lambda}} p_{\lambda}(Y) \biggr) \\
	&= \sum_{| \lambda | = k} \frac{(-1)^{l(\lambda)}}{z_{\lambda}} p_{\lambda}(Y)  \\
	&= (-1)^{k} \sum_{| \lambda | = k} \frac{(-1)^{k-l(\lambda)}}{z_{\lambda}} p_{\lambda}(Y)  \\
	&=(-1)^{k} e_{k}(Y) ,
\end{align*}
where $e_{k}$ is the elementary symmetric function of degree $k$. 
Hence, 
$$ \omega_{-Y}(s_{\nu}(Y)) = (-1)^{| \nu |} s_{\transpose{\nu}}(Y) . $$
From the well-known formula 
$$ s_{\lambda}(X,Y) = \sum_{\mu ,\, \nu} \mathrm{LR}^{\lambda}_{\mu \, \nu} s_{\mu}(X) s_{\nu}(Y) , $$ 
we obtain 
\begin{align*}
	\omega_{-Y} (s_{\lambda}(X,q^{-1}Y)) &= \sum_{\mu ,\, \nu} 
	(-1)^{|\nu|}\mathrm{LR}^{\lambda}_{\mu \, \nu} s_{\mu}(X) s_{\transpose{\nu}}(q^{-1}Y) \\
	&= \sum_{\mu ,\, \nu} (-q^{-1})^{| \nu |}\mathrm{LR}^{\lambda}_{\mu \, \nu} s_{\mu}(X) s_{\transpose{\nu}}(Y) .
\end{align*}
Therefore (\ref{lemmaformula}) follows from  
	(\ref{omega-Y}), $S_{\mu}'[j]= \iota(s_{\mu}(X))$ and $S_{\transpose{\nu}}'[j+1] = \iota(s_{\transpose{\nu}}(Y))$. 
\end{proof}

\begin{theorem}
Let $1 \le j \le \ell$ and $\lambda \in \Pi$. 
If $| \Bf{\mu} ; \Bf{s} \rangle$ is $n | \lambda |$-dominant and $\mu^{(j)}$ is $n$-restricted, 
\begin{equation*}
	S_{\lambda}[j] G^{-}(\Bf{\mu} ; \Bf{s}) 
	= G^{-}((\mu^{(1)} , \cdots , \mu^{(j)} + n \lambda, \cdots , \mu^{(\ell)}) ; \Bf{s} ) .
\end{equation*}
\label{tensorproductformula}
\end{theorem}

\begin{proof}

By definition of the basis $G^{-}$, we have to prove that $F=S_{\lambda}[j] G^{-}(\Bf{\mu} ; \Bf{s})$ satisfies 
$$
\overline{F}=F 
\quad \text{ and } \quad 
F \equiv | (\mu^{(1)} , \cdots , \mu^{(i)} + n \lambda, \cdots , \mu^{(\ell)}) ; \Bf{s} \rangle 
	\quad \mathrm{mod} \,\, q^{-1} \mathcal{L}^{-} \quad .
$$
The first property is clear by Proposition \ref{mainconj}. 
Indeed, 
$S_{\lambda}[j]$ is a $\mathbb{Q}$-linear combination of products of elements $B_{-m}[j]$. 
To prove second property, we observe that by Theorem \ref{LT.th6.7} for all $\Bf{\rho} \in \Pi^{\ell}$ and $m \in \mathbb{N}$, 
	$V_{m}'[j] \, | \Bf{\rho} ; \Bf{s} \rangle \in \mathcal{L}^{-}$. 
Thus, $S_{\lambda}'[j] \, | \Bf{\rho} ; \Bf{s} \rangle \in \mathcal{L}^{-}$ since 
	$S_{\lambda}'[j]$ is a $\mathbb{Q}$-linear combination of products of elements $V_{m}'[j]$.

Note also the following lemma. 

\begin{lemma}[\cite{LT} Proof of Theorem 6.9]
Let $\ell =1$. 
Let $\lambda , \rho$ be two partitions such that $\rho$ is $n$-restricted. 
Then, 
$$
S_{\lambda} \, | \rho \rangle \equiv | \rho + n \lambda \rangle \quad \mathrm{mod} \quad q^{-1} \mathcal{L}^{-}.
$$
\label{LTlemma2}
\end{lemma}

From the lemma, we obtain 
\begin{align*}
	F=S_{\lambda}[j] \, G^{-}(\Bf{\mu} ; \Bf{s}) 
	&\equiv \left( \sum_{\rho, \kappa} (-q^{-1})^{| \kappa |} \mathrm{LR}^{\lambda}_{\rho \, \kappa} \, 
		S_{\rho}'[j] \, S_{\transpose{\kappa}}'[j+1] \right) \, G^{-}(\Bf{\mu} ; \Bf{s}) 
		\hspace{3em} (\text{By lemma \ref{keylemma}}) \\
	&\equiv S_{\lambda}'[j] \, G^{-}(\Bf{\mu} ; \Bf{s})  
		\hspace{3em} ( \text{ By } \, S_{\transpose{\kappa}}'[j+1] \, G^{-}(\Bf{\mu} ; \Bf{s}) \in \mathcal{L}^{-} )  \\
	&\equiv S_{\lambda}'[j] \, | \Bf{\mu} ; \Bf{s} \rangle 
		\hspace{3em} ( \text{ By } \, G^{-}(\Bf{\mu} ; \Bf{s}) \equiv | \Bf{\mu} ; \Bf{s} \rangle )  \\
	& \equiv | (\mu^{(1)} , \cdots , \mu^{(i)} + n \lambda, \cdots , \mu^{(\ell)}) ; \Bf{s} \rangle 
		\hspace{3em} (\text{By lemma \ref{LTlemma2}}) .
\end{align*}

\end{proof}

\begin{definition}
For $\boldsymbol{\lambda} = (\lambda^{(1)} , \lambda^{(2)} , \cdots , \lambda^{(\ell)}) \in \Pi^{\ell}$, 
we define $\widetilde{\Bf{\lambda}}$, $\check{\Bf{\lambda}} \in \Pi^{\ell}$ by 
\begin{align*}
	\widetilde{\Bf{\lambda}} = (\widetilde{\lambda^{(1)}} , \widetilde{\lambda^{(2)}} , \cdots , \widetilde{\lambda^{(\ell)}}) 
	\quad , \quad 
	\check{\Bf{\lambda}} = (\check{\lambda^{(1)}} , \check{\lambda^{(2)}} , \cdots , \check{\lambda^{(\ell)}}) 
	\quad , 
\end{align*}
where $\lambda^{(i)} = \widetilde{\lambda^{(i)}} + n \check{\lambda^{(i)}}$ and 
	$\widetilde{\lambda^{(i)}}$ is $n$-restricted. (See Definition \ref{tilde-check}.) 
And we define 
$$
S_{\boldsymbol{\lambda}} = \prod_{i=1}^{\ell} S_{\lambda^{(i)}}[i] 
= S_{\lambda^{(1)}}[1] S_{\lambda^{(2)}}[2] \cdots S_{\lambda^{(\ell)}}[\ell] .
$$
\label{def.4.11}
\end{definition}

Now we can state our main result, 
which is a higher level version of the $q$-analogue of Theorem \ref{LTformula}. 

\begin{theorem}
If $| \Bf{\lambda} ; \Bf{s} \rangle$ is $0$-dominant, then 
\begin{equation*}
	G^{-}(\Bf{\lambda} ; \Bf{s}) 
	= S_{ \check{\Bf{\lambda}} } G^{-} (\widetilde{ \Bf{\lambda} } ; \Bf{s}) .
\end{equation*}
\label{maincor}
\end{theorem}

\begin{proof}
Note that for any partition $\mu$, $\mu = n \check{\mu} + \widetilde{\mu}$. 
Since $| \Bf{\lambda} ; \Bf{s} \rangle$ is $0$-dominant, for any $1 \le i < \ell$ and $1 \le j \le \ell$, 
\begin{align*}
	s_{i} - s_{i+1} 
	&\ge | \Bf{\lambda} | = \sum_{k=1}^{\ell} | \lambda^{(k)} | 
	\ge \sum_{k=1}^{j-1} | \widetilde{\lambda^{(k)}} | + | \lambda^{(j)} | + \sum_{k=j+1}^{\ell} | \lambda^{(k)} | 
	= n | \check{\lambda^{(j)}} | + \sum_{k=1}^{j} | \widetilde{\lambda^{(k)}} | + \sum_{k=j+1}^{\ell} | \lambda^{(k)} | .
\end{align*}
Hence $| ( \widetilde{\lambda^{(1)}} , \cdots , \widetilde{\lambda^{(j)}} , \lambda^{(j+1)} , \cdots , \lambda^{(\ell)} ; \Bf{s} ) \rangle$
is $n | \check{\lambda^{(j)}} |$-dominant for any $1 \le j \le \ell$. 
Thus, by applying Theorem \ref{tensorproductformula} repeatedly, 
\begin{align*}
	S_{\check{\Bf{\lambda}}} G^{-}(\widetilde{\Bf{\lambda}} ; \Bf{s}) 
	&= S_{\check{\lambda^{(1)}}}[1] \cdots S_{\check{\lambda^{(\ell-1)}}}[\ell-1] S_{\check{\lambda^{(\ell)}}}[\ell] 
		G^{-}((\widetilde{\lambda^{(1)}} , \cdots , \widetilde{\lambda^{(\ell-1)}} , \widetilde{\lambda^{(\ell)}} ) ; \Bf{s}) \\
	&= S_{\check{\lambda^{(1)}}}[1] \cdots S_{\check{\lambda^{(\ell-1)}}}[\ell-1] 
		G^{-}((\widetilde{\lambda^{(1)}} , \cdots , \widetilde{\lambda^{(\ell-1)}} , \lambda^{(\ell)} ) ; \Bf{s}) \\
	&= \cdots 
	= G^{-}((\lambda^{(1)} , \cdots , \lambda^{(\ell-1)} , \lambda^{(\ell)} ) ; \Bf{s}) 
	= G^{-}(\Bf{\lambda} ; \Bf{s}).
\end{align*}
\end{proof}

\begin{example}
In this example, we write $|\Bf{\lambda} ; \Bf{s} \rangle$ simply $\Bf{\lambda}$.

(i) Let $n=l=2$, $\Bf{\lambda} = ((2,2),\emptyset)$ and $\boldsymbol{s} = (2,-2)$. 
Then, $|\Bf{\lambda} ; \Bf{s} \rangle$ is $0$-dominant and  
$(2,2) = \emptyset + n (1,1)$. 
\begin{align*}
G^{-}((2,2) , \emptyset) =& S_{((1,1),\emptyset )} \, (\emptyset , \emptyset) \\
=& S_{(1,1)}[1] \, (\emptyset , \emptyset) \\
=& (S_{(1,1)}'[1] -q^{-1} S_{(1)}'[1] S_{(1)}'[2] +q^{-2} S_{(2)}'[2] ) \, (\emptyset , \emptyset)  \hspace{4em} (Lemma \ref{keylemma}) \\
=& (S_{(1,1)} \emptyset , \emptyset) - q^{-1} (S_{(1)} \emptyset , S_{(1)} \emptyset) + q^{-2} (\emptyset, S_{(2)} \emptyset) \\
=&((2,2) - q^{-1} (2,1,1) + q^{-2} (1^{4}) , \emptyset) - q^{-1} ((2) - q^{-1} (1,1) , (2) - q^{-1} (1,1)) \\
& + q^{-2} (\emptyset , (4) - q^{-1} (3,1) + q^{-2}(2,2)) \\
=& ((2,2),\emptyset) - q^{-1} ((2,1,1) , \emptyset) + q^{-2} ( (1^{4}) ,\emptyset) -q^{-1} ((2) , (2)) +q^{-2} ((2) , (1,1)) + q^{-2} ((1,1),(2)) \\
& -q^{-3} ((1,1),(1,1)) +q^{-2}(\emptyset, (4)) -q^{-3} (\emptyset , (3,1)) +q^{-4} (\emptyset , (2,2)) .
\end{align*}

(ii) Let $n=l=2$ and $\boldsymbol{s} = (3,-3)$. 
Then, $((2),(2,2))$ is $0$-dominant, $(2) = \emptyset + n (1)$ and $(2,2) = \emptyset + n (1,1)$. 
\begin{align*}
&G^{-}((2) , (2,2)) \\
=& S_{((1),(1,1) )} \, (\emptyset , \emptyset) \\
=& S_{(1)}[1] \, S_{(1,1)}[2] \, (\emptyset , \emptyset) \\
=& (S_{(1)}'[1] -q^{-1} S_{(1)}'[2] ) \, S_{(1,1)}'[2] (\emptyset , \emptyset)  \hspace{4em} (Lemma.\ref{keylemma}) \\
=& (S_{(1)}'[1] -q^{-1} S_{(1)}'[2] ) \, (\emptyset , S_{(1,1)}\emptyset) \\
=& (S_{(1)} \emptyset , S_{(1,1)}\emptyset) -q^{-1} (\emptyset , S_{(1)} S_{(1,1)}\emptyset) \\
=& (S_{(1)} \emptyset , S_{(1,1)}\emptyset) -q^{-1} (\emptyset , S_{(2,1)}\emptyset) -q^{-1} (\emptyset , S_{(1^{3})}\emptyset) \\
=& ((2) - q^{-1} , (2,2) - q^{-1} (2,1,1) + q^{-2} (1^{4}) )  \\
& -q^{-1} (\emptyset , (4,2) -q^{-1}(4,1,1) -q^{-1} (3,3) + q^{-2} (3,1^{3}) -q^{3} (2^{3}) + q^{-4} (2^{2},1^{2}) ) \\
& -q^{-1} (\emptyset , (2^{3}) -q^{-1} (2^{2},1^{2}) +q^{-2}(2,1^{4}) -q^{-3}(1^{6})) \\
=& ((2) , (2,2)) -q^{-1} ((2) , (2,1,1)) +q^{-2} ((2) , (1^{4})) -q^{-1} ((1,1) , (2,2)) +q^{-2} ((1,1) , (2,1,1)) -q^{-3} ((1,1) , (1^{4})) \\
& -q^{-1} (\emptyset , (4,2)) + q^{-2} (\emptyset , (4,1,1)) + q^{-2} (\emptyset , (3,3)) - q^{-3} (\emptyset , (3,1^{3})) 
-(q^{-1} + q^{-3}) (\emptyset , (2^{3})) \\ 
& + (q^{-2} + q^{-4}) (\emptyset , (2^{2} , 1^{2})) -q^{-3}( \emptyset , (2,1^{4})) + q^{-4} ( \emptyset , (1^{6}))  .
\end{align*}

\end{example}

%
%

\section{Proof of Proposition \ref{mainconj}}

We prove Proposition \ref{mainconj} by using (infinite) $q$-wedge product. 
Fix a sufficiently large integer $r$ so that for every ordered $q$-wedge product appearing in our argument, 
	all of the components after $r$-th factor are consecutive.
We are able to truncate $q$-wedge products at the first $r$ parts.
See \cite[\S 4]{U} for detail.

\subsection{$q$-wedges and straightening rules}

In this section, we review the straightening rules \cite{U} to prove our main results.

\begin{proposition}[\cite{I} Proposition 4.2 ,see also \cite{U} Proposition 3.16]
For unequal integers $k_{1},k_{2}$, 
	let $c_{j},d_{j},m_{j}$ be the unique integers satisfying $k_{j} = c_{j} + n (d_{j} - 1) - n \ell m_{j}$, 
$1 \le c_{j} \le n$ and $1 \le d_{j} \le \ell$, $(j=1,2)$.
Then, 

\begin{align}
u_{c_{2} - n m_{2}}^{(d_{2})} \wedge u_{c_{1} - n m_{1}}^{(d_{1})}  
	&= (-q^{-1})^{\delta_{d_{1} = d_{2} }} \,
	q^{\alpha} \, u_{c_{1} - n m_{1}}^{(d_{1})} \wedge u_{c_{2} - c m_{2}}^{(d_{2})}  
	\nonumber \\
	&+ \mathrm{sgn}(m) \,
	(-q^{-1})^{\delta_{d_{1} = d_{2} }} \, (q - q^{-1}) \sum_{j= \beta }^{ | m_{1} - m_{2} |  - \gamma} 
	u_{c_{2} - n m_{1} - \mathrm{sgn}(m) n j}^{(d_{1})} \wedge u_{c_{1} - n m_{2} + \mathrm{sgn} (m) n j}^{(d_{2})}.
\label{eq.22}
\end{align}
where 
$$\mathrm{sgn}(m) = 
\begin{cases}
1 & \text{ if } \,\, m_{1} < m_{2} \\
-1 & \text{ if } \,\, m_{1} > m_{2} \\
0 & \text{ if } \,\, m_{1} = m_{2} 
\end{cases}  \quad , \quad    
\alpha =
\begin{cases}
1 & \text{ if } \,\, c_{1} = c_{2} \text{ and } k_{1} > k_{2} \\
-1 & \text{ if } \,\, c_{1} = c_{2} \text{ and } k_{1} < k_{2} \\
0 & \text{ if } \,\, c_{1} \not= c_{2}
\end{cases} \quad , $$ 
$$ \delta_{d_{1} = d_{2}}  =
\begin{cases}
1 & \text{ if } d_{1} = d_{2} \\
0 & \text{ if } d_{1} \not= d_{2}
\end{cases} \quad , \quad 
\beta =
\begin{cases}
0 & \text{ if } \,\, c_{1} > c_{2} , m_{1} < m_{2}  \text{ or } c_{1} < c_{2} , m_{1} > m_{2}   \\
1 & \text{ if } \,\, otherwise
\end{cases} \quad , $$ 
and 
$$ \gamma =
\begin{cases}
1 & \text{ if } \,\, d_{1} < d_{2} , m_{1} < m_{2}  \text{ or } d_{1} > d_{2} , m_{1} > m_{2}   \\
0 & \text{ if } \,\, d_{1} > d_{2} , m_{1} < m_{2}  \text{ or } d_{1} < d_{2} , m_{1} > m_{2}   
\end{cases}. $$
\label{s.rule}
\end{proposition}

{\bf Remarks.} 
Note that the identity (\ref{eq.22}) depends only on the inequality relationship between $d_{1}$ and $d_{2}$
($c_{1}$ and $c_{2}$). 
It is independent of $\ell$.

\begin{corollary}
Suppose $u_{k_1}^{(d_1)} \wedge u_{k_2}^{(d_2)}$ is expressed by straightening rule as 
$$
u_{k_1}^{(d_1)} \wedge u_{k_2}^{(d_2)} = \sum_{g_1,g_2} \alpha(g_1,g_2) u_{g_2}^{(d_2)} \wedge u_{g_1}^{(d_1)} \quad ,
$$
where $\alpha(g_1,g_2) \in \mathbb{Q}(q)$. 
Then, 

(i) $k_1+k_2 = g_1+g_2$.

(ii) If $k_1<k_2$, then $k_1 \le g_1 \le k_2 $ and $k_1 \le g_2 \le k_2 $. 
If $k_1>k_2$, then $k_1 \ge g_1 \ge k_2 $ and $k_1 \ge g_2 \ge k_2 $. 

(iii) Let $k_i = c_i - nm_i$ and $g_i = c'_i - nm'_i$ ($i=1,2$) where $c_i,c'_i$ and $m_i,m'_i$ are defined 
in Definition \ref{notation}. 
Then, $\{ c_1 , c_2 \} = \{ c'_1 , c'_2 \}$. 
\label{property1}
\end{corollary}

\begin{corollary}
Let $k,m,g$ be three integers such that $m \ge 0$ and $k \ge g+nm$. 
If $u_{k}^{(d_1)} \wedge u_{g}^{(d_2)}$ and $u_{k}^{(d_1)} \wedge u_{g+nm}^{(d_2)}$ is expressed by straightening rule as 
\begin{align*}
	u_{k}^{(d_1)} \wedge u_{g}^{(d_2)} 
	= \sum_{j=0}^{k-g} C_{j}(q) \, u_{g+j}^{(d_2)} \wedge u_{k-j}^{(d_1)} \quad , \quad 
	u_{k}^{(d_1)} \wedge u_{g+nm}^{(d_2)} 
	= \sum_{j=0}^{k-g-nm} C_{j}'(q) \, u_{g+nm+j}^{(d_2)} \wedge u_{k-j}^{(d_1)} \quad ,
\end{align*}
where $C_{j}(q) , C'_{j}(q) \in \mathbb{Q}(q)$, then 
$C_{j}(q) = C'_{j}(q)$ for all $0 \le j \le k-g-nm$.
\label{property2}
\end{corollary}

\begin{definition}
Let

\begin{align*}
u &= u_{k_{1}}^{(d_{1})} \wedge u_{k_{2}}^{(d_{2})} \wedge \cdots \wedge u_{k_{r}}^{(d_{r})} \,\,\,\, , \,\,\,\,
k_{a} = c_{a} - n m_{a} \,\,\,\,,\,\,\, (a = 1, 2, \cdots , r)  \,\,\,\,  \text{ and } \\
v &= u_{g_{1}}^{(d_{1}')} \wedge u_{g_{2}}^{(d_{2}')} \wedge \cdots \wedge u_{g_{t}}^{(d_{t}')} \,\,\,\, , \,\,\,\,
g_{b} = c_{b}' - n m_{b}' \,\,\,\,,\,\,\, (b = 1, 2, \cdots , t).
\end{align*}
and suppose that $d_{a} \not= d_{b}'$ for all $a \in \{ 1, \cdots , r \}$ and $b \in \{ 1, \cdots ,t \}$. 
Then we define $\xi (u , v )$ as 

\begin{align}
\xi (u , v ) &= 
  \# \{ (a,b) \, | \, c_{a} = c_{b}'  \,\, , \,\, u_{k_{a}}^{(d_{a})} < u_{g_{b}}^{(d_{b}')}  \} .
\end{align}
\label{def-xi}
\end{definition}

\begin{lemma}[\cite{I} Lemma 4.8, (see \cite{U} Lemma 5.19)] 
Let $a \in \mathbb{Z}$, $t \in \mathbb{Z}_{\ge 0}$, $1 \le i \le \ell$, and $1 \le j \le \ell$.

$(i)$. 
Let $u_{k}^{(j)}$ be the maximal element such that $u_{k}^{(j)} < u_{a}^{(i)}$.
Let $u_{[k,k-t]}^{(j)} = u_{k}^{(j)} \wedge u_{k-1}^{(j)} \wedge \cdots \wedge u_{k-t}^{(j)}$.
Then, 

\begin{equation*}
u_{a}^{(i)} \wedge u_{[k,k-t]}^{(j)}
=q^{- \xi (u_{[k,k-t]}^{(j)} , u_{a}^{(i)})} u_{[k,k-t]}^{(j)} \wedge u_{a}^{(i)}.
\end{equation*}

$(ii)$. 
Let $u_{g}^{(j)}$ be the minimal element such that $u_{g}^{(j)} > u_{a}^{(i)}$.
Let $u_{[g+t,g]}^{(j)} = u_{g+t}^{(j)} \wedge u_{g+t-1}^{(j)} \wedge \cdots \wedge u_{g}^{(j)}$.
Then, 

\begin{equation*}
u_{a}^{(i)} \wedge u_{[g+t,g]}^{(j)}
=q^{\xi (u_{a}^{(i)} , u_{[g+t,g]}^{(j)})} u_{[g+t,t]}^{(j)} \wedge u_{a}^{(i)}.
\end{equation*}

\label{lem3}
\end{lemma}

In the abacus presentation, $u_{a}^{(i)} , u_{[k,k-t]}^{(j)}$ and $u_{[g+t,g]}^{(j)}$ look as follows.

$$\begin{array}{cccc|crccc}
\mathrm{(i)} & \multicolumn{3}{c}{ d=i } & \multicolumn{4}{c}{ d=j } \\ \cline{8-9}
& & & &  &  & & \multicolumn{1}{|c}{ \small \centercircle{k-t} \normalsize } & \multicolumn{1}{c|}{} \\ \cline{6-7}
& & & &  & \multicolumn{1}{|c}{ \hspace{1em} \multirow{3}{*}{ $\xi  \, \left\{  \begin{array}{c} \bullet \\ \vdots \\ \bullet \end{array}   \right.$ } } & & & \multicolumn{1}{c|}{} \\ 
& & & &  & \multicolumn{1}{|c}{} & & & \multicolumn{1}{c|}{} \\
& & & &  & \multicolumn{1}{|c}{} & & \small \centercircle{k-1} \normalsize & \multicolumn{1}{c|}{ \centercircle{k} } \\  \cline{6-9}
& \centercircle{a} & &  &  &  | \hspace{1em}  & & \\
& & & &  & c \equiv a & & &  
\end{array} 
\hspace{1em} , 
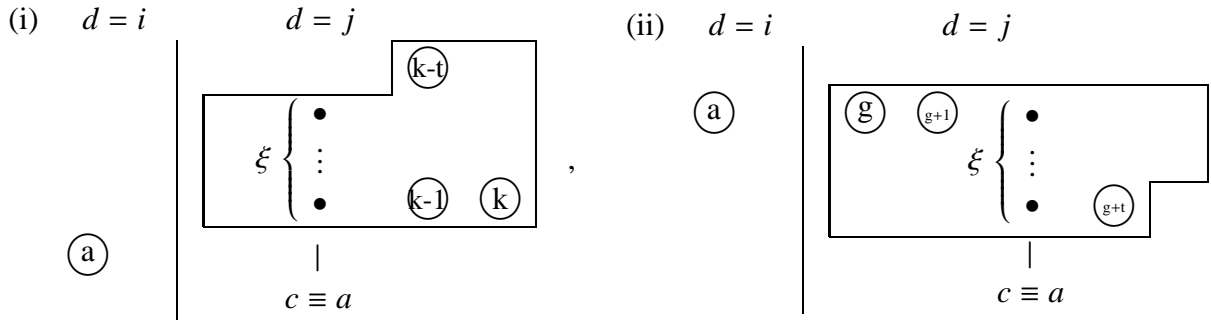
\hspace{1em}
\begin{array}{cccc|ccrcc}
\mathrm{(ii)} & \multicolumn{3}{c}{ d=i } & \multicolumn{4}{c}{ d=j } \\ 
& & & &  & & & & \\ \cline{6-9}
& \centercircle{a} & &  & &  \multicolumn{1}{|c}{ \centercircle{g} } &  \tiny \centercircle{g+1} \normalsize 
  \multirow{3}{*}{ $\xi  \, \left\{  \begin{array}{c} \bullet \\ \vdots \\ \bullet \end{array}   \right.$ }  
  & & \multicolumn{1}{c|}{\hspace{1em}} \\  
& & & &  & \multicolumn{1}{|c}{} & & & \multicolumn{1}{c|}{} \\ \cline{9-9}
& & & &  & \multicolumn{1}{|c}{} & & \multicolumn{1}{c|}{ \tiny \centercircle{g+t} \normalsize } &  \\ \cline{6-8}
& & & &  & & | \hspace{1em} & & \\ 
& & & &  & & c \equiv a & & 
\end{array}
$$
where the boxed region means that all positions are occupied by beads.

%
%

Throughout this section, $\Bf{s} = (s_{1} , \ldots , s_{\ell}) \in \mathbb{Z}^{\ell}$ is a fixed multi charge. 

\subsection{A deletion lemma for finite $q$-wedges}

\begin{definition}
Let $1 \le d_{1} , d_{2} \le \ell$. 
Let 
$v^{(d_{1})} = u_{k_{1}}^{(d_{1})} \wedge u_{k_{2}}^{(d_{1})} \wedge \cdots \wedge u_{k_{r}}^{(d_{1})} $ and 
$w^{(d_{2})} = u_{g_{1}}^{(d_{2})} \wedge u_{g_{2}}^{(d_{2})} \wedge \cdots \wedge u_{g_{R}}^{(d_{2})} $ 
be two simple finite $q$-wedges. 

(i) The {\it size} of $v^{(d_{1})}$ for $s_{d_{1}}$ is 
$$ | v^{(d_{1})} |_{s_{d_{1}}} = 
	| u_{k_{1}}^{(d_{1})} \wedge u_{k_{2}}^{(d_{1})} \wedge \cdots \wedge u_{k_{r}}^{(d_{1})} |_{s_{d_{1}}} 
	= \sum_{i=1}^{r} | k_{i} - s_{d_{1}} + i | . $$

(ii) For a non-negative integer $M$, the pair $(v^{(d_{1})} , w^{(d_{2})} , s_{d_{1}} , s_{d_{2}})$ is 
	$M$-{\it dominant} if 
$$ | v^{(d_{1})} |_{s_{d_{1}}} +  | w^{(d_{2})} |_{s_{d_{2}}} + M \le s_{d_{1}} - s_{d_{2}} . $$
\label{def.6.1}
\end{definition}

\begin{example}
$(i)$ 
Let $s_{1} = 2$, $s_{2} = -4$, 
$v^{(1)} = u_{4}^{(1)} \wedge u_{2}^{(1)} \wedge u_{1}^{(1)} \wedge u_{-1}^{(1)} \wedge u_{-2}^{(1)} \wedge u_{-3}^{(1)}$ and 
$w^{(2)} = u_{-2}^{(2)}$. \\
Then $ | v^{(1)} |_{s_{1}} = 4 $ and $ | w^{(2)} |_{s_2} = 2 $. 
Thus, $| v^{(1)} |_{s_{1}} +  | w^{(2)} |_{s_{2}} \le s_{1} - s_{2}$. 
Therefore $(v^{(1)} , w^{(2)} , s_{1} , s_{2})$ is $0$-dominant.  

$(ii)$ Let $r$ be a positive integer and $\lambda$ be a partition whose length is less than $r$. 
For any $1 \le j \le \ell$, put 
$$ v_{\lambda}^{(j)} = u_{\lambda_{1}+s_{j}}^{(j)} \wedge u_{\lambda_{2}+s_{j}-1}^{(j)} \wedge \cdots 
\wedge u_{\lambda_{r}+s_{j}-r+1}^{(j)}. $$
Then, $| v_{\lambda}^{(j)} |_{s_{j}} = | \lambda | $. 
\end{example}

{\bf Remarks}. 
(i). Let $v^{(j)} = u_{k_{1}}^{(j)} \wedge u_{k_{2}}^{(j)} \wedge \cdots \wedge u_{k_{r}}^{(j)} $ and 
$w^{(j)} = u_{g_{1}}^{(j)} \wedge u_{g_{2}}^{(j)} \wedge \cdots \wedge u_{g_{r}}^{(j)} $ be two simple finite $q$-wedges 
such that the length of $v^{(j)}$ is equal to that of $w^{(j)}$. 
Then, the size of $v^{(j)}$ for $s_{j}$ is equal to that of $w^{(j)}$ if and only if 
$\sum_{i=1}^{r} k_{i} = \sum_{i=1}^{r} g_{i}$. 

(ii) From Corollary \ref{property1} (i), straightening rule preserves 
$ | v^{(d_{1})} |_{s_{d_{1}}} +  | w^{(d_{2})} |_{s_{d_{2}}} $. 
That is, if $x^{(d_{1})} \wedge y^{(d_{2})}$ appears in a linear expansion of $w^{(d_{2})} \wedge v^{(d_{1})}$, 
then 
$| v^{(d_{1})} |_{s_{d_{1}}} +  | w^{(d_{2})} |_{s_{d_{2}}} = | x^{(d_{1})} |_{s_{d_{1}}} +  | y^{(d_{2})} |_{s_{d_{2}}} $.

In particular, for a non-negative integer $M$, if the pair $(v^{(d_{1})} , w^{(d_{2})} , s_{d_{1}} , s_{d_{2}})$ is $M$-dominant, 
then the pair $(x^{(d_{1})} , y^{(d_{2})} , s_{d_{1}} , s_{d_{2}})$ is also $M$-dominant.

\begin{lemma}
Let $1 \le j \le \ell$ and 
$v^{(j)} = u_{k_{1}}^{(j)} \wedge u_{k_{2}}^{(j)} \wedge \cdots \wedge u_{k_{r}}^{(j)}$ be a simple finite $q$-wedge.  
Suppose that $\min \{ k_{1} , k_{2} , \cdots , k_{r} \} \ge s_{j} -r+1$ and $ | v^{(j)} |_{s_{j}} < 0 $, then $ v^{(j)} = 0 $. 
\label{lemma1}
\end{lemma}

\begin{proof}
If $ v^{(j)} \ne 0 $, then an ordered simple finite $q$-wedge  
$v' = u_{k'_{1}}^{(j)} \wedge u_{k'_{2}}^{(j)} \wedge \cdots \wedge u_{k'_{r}}^{(j)} $ 
	appears in the linear expansion of $v^{(j)}$. 
Then from Corollary \ref{property1} (ii), we have 
$ k'_{1} > k'_{2} > \cdots > k'_{r} \ge s_{j} -r+1 $. Thus $ | v' |_{s_{j}} \ge 0 $. 

On the other hand, from Corollary \ref{property1} (i) (see the previous Remark (ii)), 
$ | v^{(1)} |_{s_{1}} = | v' |_{s_{1}} $. 
This is a contradiction. 
\end{proof}

\begin{lemma}
Let $M,\gamma \in \mathbb{Z}_{\ge 0}$ such that $0 \le \gamma \le M$. 
Let $1 \le d_{1},d_{2} \le \ell, t \in \mathbb{Z}$ and 
$v^{(d_{1})} = u_{k_{1}}^{(d_{1})} \wedge u_{k_{2}}^{(d_{1})} \wedge \cdots \wedge u_{k_{r}}^{(d_{1})} $. 
If $(v^{(d_{1})} , u_{t}^{(d_{2})} , s_{d_{1}} , s_{d_{2}})$ is $M$-dominant, then 
$$ | u_{t + \gamma}^{(d_{1})} \wedge v^{(d_{1})} |_{s_{d_{1}}+1} < 0 .$$ 

Moreover, if $\min \{ t+\gamma , k_{1},k_{2},\cdots,k_{r} \} \ge s_{d_{1}} - r$, 
	then $u_{t + \gamma}^{(d_{1})} \wedge v^{(d_{1})} = 0$. 
\label{lemma2}
\end{lemma}

\begin{proof}
For simplicity, we assume that $d_{1} = 1$ and $d_{2} = 2$. 
\begin{align*}
	| u_{t + \gamma}^{(1)} \wedge v^{(1)} |_{s_{1}+1} 
	&= t+\gamma - (s_{1}+1) + \sum_{i=1}^{r} (k_{i} - (s_{1}+1) + (i+1) - 1) \\
	&= t-s_{1}+\gamma-1 + | v^{(1)} |_{s_{1}} \\
	&= |u_{t}^{(2)}|_{s_{2}} + | v^{(1)} |_{s_{1}} + s_{2} - s_{1} + \gamma - 1 
		\hspace{3em} (\text{ By } |u_{t}^{(2)}|_{s_{2}} = t - s_{2} ) \\
	&\le s_{1} - s_{2} - M + s_{2} - s_{1} + \gamma - 1 
		\hspace{3em} (\text{ $(v^{(1)} , u_{t}^{(2)} , s_{1} , s_{2})$ is $M$-dominant }) \\
	&= \gamma - M - 1 < 0 . \hspace{3em} ( \gamma \le M )
\end{align*}
The second statement follows from the first statement and Lemma \ref{lemma1}. 
\end{proof}

\begin{lemma}
Let $v^{(j)} = u_{k_{1}}^{(j)} \wedge u_{k_{2}}^{(j)} \wedge \cdots \wedge u_{k_{r}}^{(j)}$ 
	be an ordered $q$-wedge such that $k_{r} \ge s_{j} -r+1$. 
Then, for any $1 \le a \le r$, 
$$ | u_{k_{a}}^{(j)} \wedge u_{k_{i+1}}^{(j)} \wedge \cdots \wedge u_{k_{r}}^{(j)} |_{s_{j} -a+1} 
	\le | v^{(j)} |_{s_{j}} . $$

In particular, for a non-negative integer $M$, if $(v^{(d_{1})} , w^{(d_{2})} , s_{d_{1}} , s_{d_{2}})$ is $M$-dominant, 
	then for all $1 \le a \le r$, 
	$(u_{k_{a}}^{(d_{1})} \wedge u_{k_{a+1}}^{(d_1)} \wedge \cdots \wedge u_{k_{r}}^{(d_1)} , w^{(2)} , s_{d_1}-i+1 , s_{d_2})$ 
	is also $M$-dominant. 
\label{lemma3}
\end{lemma}

\begin{proof}
Since $v^{(j)}$ is ordered and $k_{r} \ge s_{j}-r+1$, $k_{1} > k_{2} > \cdots > k_{r} \ge s_{j} -r+1$. 
Hence, for any $1 \le a \le r$, $k_{a} \ge s_{j}-a+1$. 
Thus, 
\begin{align*}
	| u_{k_{a}}^{(j)} \wedge u_{k_{a+1}}^{(j)} \wedge \cdots \wedge u_{k_{r}}^{(j)} |_{s_{j}-a+1} 
	&= \sum_{i=a}^{r} (k_i -(s_j-a+1)+(i-a+1)-1) \\
	&= \sum_{i=a}^{r} (k_i - s_j + i -1) \\
	&\le \sum_{i=1}^{r} (k_i - s_j + i -1) 
	= |v^{(j)}|_{s_{j}} 
\end{align*}

The second statement is clear. 
\end{proof}

\subsection{A exchange rule for dominant finite $q$-wedges}

\begin{definition}
Let $1 \le j \le \ell$, $\lambda \in \Pi$ and $k,g$ be two integers such that $k \ge g$. 
We define 
\begin{align*}
	u_{[k,g]}^{(j)} &= u_{k}^{(j)} \wedge u_{k-1}^{(j)} \wedge u_{k-2}^{(j)} \wedge \cdots \wedge u_{g+1}^{(j)} \wedge u_{g}^{(j)} , \\
	v_{\lambda,r}^{(j)} &= u_{\lambda_{1}+s_{j}}^{(j)} \wedge u_{\lambda_{2}+s_{j}-1}^{(j)} \wedge 
		\cdots \wedge u_{\lambda_{r}+s_{j}-r+1}^{(j)} .
\end{align*}
\end{definition}

\begin{lemma}
Let $m \in \mathbb{Z}_{\ge 0}$ and $\lambda \in \Pi$. 
Let $t,t_0$ be two integers such that $t \ge t_{0}$. 
Suppose that $ | \lambda | + | u_{t}^{(d_{2})} |_{s_{d_{2}}} \le s_{d_{1}} - s_{d_{2}} - nm $. 
Put $r = s_{d_{1}} - t_{0}$ and $v^{(d_{1})} = v_{\lambda,r}^{(d_{1})}$. 
If $v^{(d_1)} \wedge u_{t}^{(d_2)}$ is expressed as 
\begin{align*}
	v^{(d_1)} \wedge u_{t}^{(d_2)} 
	&= \sum_{t' , \Bf{y}} \alpha(t',y) \, u_{t'}^{(d_2)} \wedge y \quad , 
\end{align*}
where $t' \in \mathbb{Z}$, $y = u_{\Bf{k}}^{(d_1)} \, (\Bf{k} \in \mathbb{Z}^{r})$ and $ \alpha(t',y) \in \mathbb{Q}(q) $. 
Then, 
\begin{align*}
	v^{(d_1)} \wedge u_{t+nm}^{(d_2)} 
	&= q^{2m} \sum_{t' , y} \alpha(t',y) \, u_{t'+nm}^{(d_2)} \wedge y \quad .
\end{align*}
\label{lemma4}
\end{lemma}

\begin{example}
Let $n=\ell=2, m=1$ and $s_1=2 , s_2=-4$. 
Let $\lambda = (1,1)$ and $t=t_{0}=-4$. 
Then, $r=6$ and 
$v^{(1)} = v_{(1,1),6}^{(1)} 
	= u_{2}^{(1)} \wedge u_{1}^{(1)} \wedge u_{[-1,-4]}^{(1)} 
	= u_{2}^{(1)} \wedge u_{1}^{(1)} \wedge u_{-1}^{(1)} \wedge u_{-2}^{(1)} \wedge u_{-3}^{(1)} \wedge u_{-4}^{(1)}$. 

Since $|\lambda| + |u_{-4}^{(2)}|_{s_2} = 2+0 \le s_1 - s_2 - n m$, these satisfy the condition in Lemma \ref{lemma4}. 
The expansions of $v^{(1)} \wedge u_{-4}^{(2)}$ and $v^{(1)} \wedge u_{-2}^{(2)}$ are 
\begin{align*}
	v^{(1)} \wedge u_{-4}^{(2)} 
	=& q^{-1} u_{-4}^{(2)} \wedge u_{2}^{(1)} \wedge u_{1}^{(1)} \wedge u_{[-1,-4]}^{(1)} 
	- (q-q^{-1}) u_{-3}^{(2)} \wedge u_{2}^{(1)} \wedge u_{0}^{(1)} \wedge u_{[-1,-4]}^{(1)} \\
	& + (q^{2}-1)u_{-2}^{(2)} \wedge u_{1}^{(1)} \wedge u_{0}^{(1)} \wedge u_{[-1,-4]}^{(1)} \quad , \\
	v^{(1)} \wedge u_{-2}^{(2)} 
	=& q u_{-2}^{(2)} \wedge u_{2}^{(1)} \wedge u_{1}^{(1)} \wedge u_{[-1,-4]}^{(1)} 
	- (q^3-q) u_{-1}^{(2)} \wedge u_{2}^{(1)} \wedge u_{0}^{(1)} \wedge u_{[-1,-4]}^{(1)} \\
	& + (q^{4}-q^{2})u_{0}^{(2)} \wedge u_{1}^{(1)} \wedge u_{0}^{(1)} \wedge u_{[-1,-4]}^{(1)} . 
\end{align*}

On the other hand, let $w^{(1)} = u_{2}^{(1)} \wedge u_{1}^{(1)}$. 
Then, the expansions of $w^{(1)} \wedge u_{-4}^{(2)}$ and $w^{(1)} \wedge u_{-2}^{(2)}$ are 
\begin{align*}
	w^{(1)} \wedge u_{-4}^{(2)} 
	=& q^{-1} u_{-4}^{(2)} \wedge u_{2}^{(1)} \wedge u_{1}^{(1)} 
	-(q-q^{-1}) u_{-3}^{(2)} \wedge u_{2}^{(1)} \wedge u_{0}^{(1)} 
	+ (q^{2}-1)u_{-2}^{(2)} \wedge u_{1}^{(1)} \wedge u_{0}^{(1)} \\
	& -(q-q^{-1}) u_{-1}^{(2)} \wedge u_{2}^{(1)} \wedge u_{-2}^{(1)} 
	+(q^{2}+1) u_{0}^{(2)} \wedge u_{1}^{(1)} \wedge u_{-2}^{(1)} \quad ,\\	
	w^{(1)} \wedge u_{-2}^{(2)} 
	=& q^{-1} u_{-2}^{(2)} \wedge u_{2}^{(1)} \wedge u_{1}^{(1)} 
	-(q-q^{-1}) u_{-1}^{(2)} \wedge u_{2}^{(1)} \wedge u_{0}^{(1)} 
	+ (q^{2}-1)u_{0}^{(2)} \wedge u_{1}^{(1)} \wedge u_{0}^{(1)} . 
\end{align*}

In general, the expansion of $w^{(1)} \wedge u_{t}^{(2)}$ has more terms than that of $w^{(1)} \wedge u_{t+nm}^{(2)}$. 
Lemma \ref{lemma4} asserts that $u_{[-1,-4]}^{(1)}$ deletes the "excessive" terms 
and the coefficient $q^{2m}$ comes from the difference between 
$u_{[-1,-4]}^{(1)} \wedge u_{-4}^{(2)}$ and $u_{[-1,-4]}^{(1)} \wedge u_{-2}^{(2)}$. 
\end{example}

\begin{proof}[(proof of Lemma \ref{lemma4})]
For simplicity, we assume that $d_{1} = 1$ and $d_{2} = 2$. 

From $r=s_{1}-t_{0}$ and $| \lambda | + | u_{t}^{(2)} |_{s_2} \le s_{1}-s_{2}-nm$, we have  
$| \lambda | \le r - nm - t + t_{0}$. 
Thus, $\lambda_{r} = \lambda_{r-1} = \cdots = \lambda_{r-nm-t+t_{0}+1} = 0$ and $|\lambda| = |v^{(1)}|_{s_1}$. 
In particular, the pair $(v^{(1)} , u_t^{(2)} , s_1 , s_2)$ is $nm$-dominant. 

We divide $v^{(1)}$ into two parts, say 
\begin{align*}
	w^{(1)} &= u_{\lambda_{1} + s_{1}}^{(1)} \wedge u_{\lambda_{2} + s_{1} - 1}^{(1)} \wedge 
	\cdots \wedge u_{\lambda_{r-nm-t+t_{0}} +s_{1}-(r-nm-t+t_{0}-1)}^{(1)} 
	= v_{\lambda , r-nm-t+t_{0}}^{(1)}
\quad , \quad \\
	x^{(1)} &= u_{s_{1}-r+nm+t-t_{0}}^{(1)} \wedge u_{s_{1}-r+nm+t-t_{0}-1}^{(1)} \wedge \cdots \wedge u_{s_{1}-r+1}^{(1)}
	= u_{[t+nm,t_{0}+1]}^{(1)} 
\quad .
\end{align*}
It is clear that $v^{(1)} = w^{(1)} \wedge x^{(1)}$. 

Let $c = \xi (u_{[t,t_{0}+1]}^{(1)} , u_{t}^{(2)})$. 
By Lemma \ref{lem3}, 
$x^{(1)} \wedge u_{t}^{(2)} = q^{c-m} u_{t}^{(2)} \wedge x^{(1)}$ and 
$x^{(1)} \wedge u_{t+nm}^{(2)} = q^{c+m} u_{t}^{(2)} \wedge x^{(1)}$, 
since $\xi (u_{[t+nm,t+1]}^{(1)} , u_{t}^{(2)}) = \xi (u_{t+nm}^{(2)} , u_{[t+nm,t+1]}^{(1)}) = m$. 
Hence, the following claim deduces Lemma \ref{lemma4}.
\smallskip

{\bf Claim 1.} \it 
If $w^{(1)} \wedge u_{t}^{(2)} \wedge x^{(1)}$ is expressed as 
\begin{align*}
	w^{(1)} \wedge u_{t}^{(2)} \wedge x^{(1)} 
	&= \sum_{t' , \Bf{y}} \beta(t',y) \, u_{t'}^{(2)} \wedge y \quad , 
\end{align*}
where $t' \in \mathbb{Z}$, $y = u_{\Bf{k}}^{(1)} \, (\Bf{k} \in \mathbb{Z}^{r})$ and $ \beta(t',y) \in \mathbb{Q}(q) $. 
Then, 
\begin{align*}
	w^{(1)} \wedge u_{t+nm}^{(2)} \wedge x^{(1)} 
	&= \sum_{t' , y} \beta(t',y) \, u_{t'+nm}^{(2)} \wedge y \quad .
\end{align*}

\rm

\begin{proof}[(proof of Claim 1.)]
We prove it by induction on the size of the partition $\lambda$. 

If $|\lambda| =0$, we have $w^{(1)} = u_{[s_{1},t+nm+1]}^{(1)}$. 
Thus, from Lemma \ref{lem3}, 
\begin{align*}
	w^{(1)} \wedge u_{t}^{(2)} \wedge x^{(1)} 
	&= q^{-c+m} u_{[s_{1},t+nm+1]}^{(1)} \wedge x^{(1)} \wedge u_{t}^{(2)} 
	= q^{-c+m} u_{[s_{1},t_{0}+1]}^{(1)} \wedge u_{t}^{(2)} \\
	&= q^{c'} u_{t}^{(2)} \wedge u_{[s_{1},t_{0}+1]}^{(1)} \quad , 
\end{align*}
where $c' = \xi( u_{[s_{1},t+nm+1]}^{(1)} , u_{t}^{(2)} )$. 
On the other hand, from $\xi( u_{[s_{1},t+nm+1]}^{(1)} , u_{t+nm}^{(2)} ) = c'$ and Lemma \ref{lem3}, 
\begin{align*}
	w^{(1)} \wedge u_{t+nm}^{(2)} \wedge x^{(1)} 
	= u_{[s_{1},t+nm+1]}^{(1)} \wedge u_{t+nm}^{(2)} \wedge x^{(1)} 
	&= q^{c'} u_{t+nm}^{(2)} \wedge u_{[s_{1},t_{0}+1]}^{(1)} \quad . 
\end{align*}

Suppose that $|\lambda|>0$. 
Let $\check{w}^{(1)}$ be the simple finite $q$-wedge removing the first component from $w^{(1)}$, i.e. 
$$\check{w}^{(1)} = u_{\lambda_{2} + s_{1} - 1}^{(1)} \wedge u_{\lambda_{3} + s_{1} - 2}^{(1)} \wedge 
	\cdots \wedge u_{\lambda_{r-nm-t+t_{0}} +s_{1}-(r-nm-t+t_{0}-1)}^{(1)} . $$
Note that 
$ |\check{w}^{(1)}|_{s_{1}-1} = |w^{(1)}|_{s_{1}-1} - \lambda_1 = |\lambda| - \lambda_1 .$
Since $| \lambda | + | u_{t}^{(2)} |_{s_2} \le s_{1}-s_{2}-nm$ and $\lambda_1 \ge 1$, 
\begin{equation}
	|\check{w}^{(1)}|_{s_{1}-1} + | u_{t}^{(2)} |_{s_2}
	= | \lambda | - \lambda_1 + | u_{t}^{(2)} |_{s_2} 
	\le (s_{1}-1)-s_{2}-nm \quad . 
\label{eq.26}
\end{equation} 
Thus, the pair $\{ \check{\lambda} = (\lambda_2 , \lambda_3 , \cdots ), s_{1}-1 , s_{2} , t , t_{0} \}$ satisfy the condition of the claim 1.
Hence, from the induction hypothesis, 
if $\check{w}^{(1)} \wedge u_{t}^{(2)} \wedge x^{(1)}$ is expressed as 
$\check{w}^{(1)} \wedge u_{t}^{(2)} \wedge x^{(1)} = \sum_{t' , \Bf{y}} \check{\beta}(t',y) \, u_{t'}^{(2)} \wedge y $
then, 
\begin{align}
	\check{w}^{(1)} \wedge u_{t+nm}^{(2)} \wedge x^{(1)} 
	&= \sum_{t' , y} \check{\beta}(t',y) \, u_{t'+nm}^{(2)} \wedge y \quad .
\label{eq.27}
\end{align}
From Corollary \ref{property1} (i) and (ii), for all $t', y=u_{\Bf{k}}^{(1)} = u_{k_{2}}^{(1)} \wedge \cdots \wedge u_{k_{r}}^{(1)}$ 
such that $\check{\beta}(t',y) \ne 0$, 
\begin{equation}
	|\check{w}^{(1)}|_{s_{1}-1} + | u_{t}^{(2)} |_{s_2} = |y|_{s_{1}-1} + | u_{t'}^{(2)} |_{s_2} 
	\quad \text{ and } \quad 
	\lambda_{2} + s_1 -1 \ge t'+nm \ge t +nm \quad . 
\label{eq.28}
\end{equation}

Let $a=\lambda_1+s_1$. 
From Corollary \ref{property1} (ii), for $t' \le a-nm$
\begin{align}
	u_{a}^{(1)} \wedge u_{t'}^{(2)} 
	&= \sum_{j=0}^{a-t'} C_{j}(q) \, u_{t'+j}^{(2)} \wedge u_{a-j}^{(1)} \quad , 
\label{eq.29} \\
	u_{a}^{(1)} \wedge u_{t'+nm}^{(2)} 
	&= \sum_{j=0}^{a-t'-nm} C_{j}'(q) \, u_{t'+nm+j}^{(2)} \wedge u_{a-j}^{(1)} \quad ,
\label{eq.30}
\end{align}
where $C_{j}(q) \in \mathbb{Q}(q)$. 
Note that if $0 \le j \le a-t'-nm$ then $C_{j}(q) = C_{j}'(q)$ because of Corollary \ref{property2}. 
From (\ref{eq.27}), (\ref{eq.29}) and (\ref{eq.30}), we have
\begin{align*}
	w^{(1)} \wedge u_{t}^{(2)} \wedge x^{(1)} 
	&= \sum_{t' , y} \sum_{j=0}^{a-t'} C_{j}(q) \check{\beta}(t',y) \, u_{t'+j}^{(2)} \wedge u_{a-j}^{(1)} \wedge y \quad , \\
	w^{(1)} \wedge u_{t+nm}^{(2)} \wedge x^{(1)} 
	&= \sum_{t' , y} \sum_{j=0}^{a-t'-nm} C_{j}(q) \check{\beta}(t',y) \, u_{t'+nm+j}^{(2)} \wedge u_{a-j}^{(1)} \wedge y \quad . 
\end{align*}
Therefore the following claim proves Claim 1. and Lemma \ref{lemma4}. 
\smallskip

{\bf Claim 2. } \it 
Let $a-t'-nm+1 \le j \le a-t'$. Then $u_{a-j}^{(1)} \wedge y = 0$. 
\rm

\begin{proof}[(proof of Claim 2.)]
Let $\gamma =a-t'-j$. 
Then, $t'+\gamma = a-j$ and $0 \le \gamma < nm$.
We exchange $j$ for $\gamma$. 
i.e. we prove that $u_{t'+\gamma}^{(1)} \wedge y = 0$ if $0 \le \gamma < nm$. 
To prove it, we use Lemma \ref{lemma2}. 

Note that the length of the finite $q$-wedge $u_{t'+\gamma}^{(1)} \wedge y$ is equal to that of $v^{(1)}$, that is $r=s_{1}-t_{0}$. 
From Corollary \ref{property1} (ii), it is clear that 
	$\min \{ t'+\gamma , k_{2},k_{3},\cdots,k_{r} \} \ge t_{0}+1 = s_1-(r-1)$, 
where $y=u_{\Bf{k}}^{(1)} = u_{k_{2}}^{(1)} \wedge \cdots \wedge u_{k_{r}}^{(1)}$. 
Thus, it is enough that we prove the pair $(y,u_{t'}^{(2)},s_{1}-1,s_{2})$ is $nm$-dominant. 
From (\ref{eq.26}) and (\ref{eq.28}), 
\begin{align*}
	|y|_{s_{1}-1} + |u_{t'}^{(2)}|_{s_2}
	= |\check{w}^{(1)}|_{s_{1}-1} + |u_{t}^{(2)}|_{s_2} 
	\le s_1 - s_2 -nm .
\end{align*}

\end{proof}
\end{proof}
\end{proof}

\subsection{Bar involution and dominant finite $q$-wedges}

\begin{definition}
Let $v = u_{\Bf{k}} = u_{k_{1}} \wedge u_{k_{2}} \wedge \cdots \wedge u_{k_{r}}$ be a finite $q$-wedge. 
We define a finite $q$-wedge $\overset{\leftarrow}{v}$ as 
$$ 
\overset{\leftarrow}{v} = u_{k_{r}} \wedge u_{k_{r-1}} \wedge \cdots \wedge u_{k_{2}} \wedge u_{k_{1}} . 
$$
\end{definition}

\begin{definition}
Let $1 \le d_{1} < d_{2} \cdots < d_{b} < j < d_{b+1} < \cdots < d_{c} \le \ell$ and 
$v^{(d_{i})} = u_{k^{(i)}_{1}} \wedge u_{k^{(i)}_{2}} \wedge \cdots$ ($1 \le i \le c$) be simple finite $q$-wedges. 
We define the action of $B'_{-m}[j]$ on the finite $q$-wedge 
$v=v^{(d_{1})} \wedge v^{(d_{2})} \wedge \cdots \wedge v^{(d_{b})} \wedge v^{(j)} 
	\wedge v^{(d_{b+1})} \wedge \cdots \wedge v^{(d_{c})}$ as
\begin{equation*}
	B'_{-m}[j] \, v 
	= q^{-bm} v^{(d_{1})} \wedge v^{(d_{2})} \wedge \cdots \wedge v^{(d_{b})} \wedge (B_{-m} \, v^{(j)} ) 
		\wedge v^{(d_{b+1})} \wedge \cdots \wedge v^{(d_{c})} \quad , 
\end{equation*}
where $B_{-m} \, v^{(j)}$ means the level one action of $B_{-m}$ on $v^{(j)}$. i.e. if 
$v^{(j)} = u_{k_{1}}^{(j)} \wedge u_{k_{2}}^{(j)} \wedge \cdots \wedge u_{k_{r}}^{(j)}$, then 
$$
	B_{-m} \, v^{(j)} = \sum_{i=1}^{r} u_{k_{1}}^{(j)} \wedge \cdots \wedge u_{k_{i+nm}}^{(j)} \wedge \cdots \wedge u_{k_{r}}^{(j)} .
$$ 

We also define $B_{-m}[j,\ell]$ on finite $q$-wedges as similarly in Definition \ref{defB}. i.e. 
\begin{align}
	B_{-m}[j,\ell] = \sum_{i = j}^{\ell} q^{(i-j)m} B_{-m}'[i] .
\end{align}
\label{def.6.10}
\end{definition}

{\bf Remark.} 
From \cite[Proposition 5.3.(i)]{U}, if $| \Bf{\lambda} ; \Bf{s} \rangle$ is $nm$-dominant, 
	then the above definition of $B_{-m}'[j]$ and $B_{-m}[j,\ell]$ on finite $q$-wedges 
	coincides with that on $| \Bf{\lambda} ; \Bf{s} \rangle$. 

\begin{lemma}
Let $m \ge 0$ and $1 \le d_{1} < \cdots < d_{b} < j$. 
Let $t,t_{0}$ be two integers such that $t \ge t_{0}$. 
Let $\lambda^{(d_1)} , \lambda^{(d_2)} ,\cdots , \lambda^{(d_b)} \in \Pi$ such that 
$|\lambda^{(d_1)}| + \cdots + |\lambda^{(d_b)}| + |u_{t}^{(j)}|_{s_j} \le s_{d_{i}}-s_j -nm$ for all $1 \le i \le b$. 
For $1 \le i \le b$, put $r_i = s_{d_i}-t_0$ and 
$$
	v^{(d_{i})} = v^{(d_{i})}_{\lambda^{(d_i)} , r_i} 
	= u_{\lambda_1^{(d_i)}+s_{d_i}}^{(d_i)} \wedge u_{\lambda_2^{(d_i)}+s_{d_i}-1}^{(d_i)} \wedge \cdots 
	\wedge u_{\lambda_{r_i}^{(d_i)}+s_{d_i}-r_i+1}^{(d_i)} . 
$$
Set $v = v^{(d_{1})} \wedge v^{(d_{2})} \wedge \cdots v^{(d_{b})}$. 
If $v \wedge u_{t}^{(j)}$ is expressed as 
\begin{align}
	v \wedge u_{t}^{(j)} 
	&= \sum_{t' , y} \alpha(t',y) \, u_{t'}^{(j)} \wedge y \quad , 
\label{eq.33}
\end{align}
where $t' \in \mathbb{Z}$, 
$y = y^{(d_1)} \wedge \cdots \wedge y^{(d_b)}$ 
(for all $1 \le i \le b$, $y^{(d_i)} = u_{\Bf{k}^{(d_i)}}^{(d_i)} \, (\Bf{k}^{(i)} \in \mathbb{Z}^{r_i}))$ 
	and $ \alpha(t',y) \in \mathbb{Q}(q) $. 
Then, 
\begin{itemize}
\item[(i)] For all $1 \le i \le b$, $|v^{(d_i)}|_{s_i} = |\lambda^{(d_i)}|$.
\item[(ii)] For all $t'$ and $y$, if $\alpha(t',y) \ne 0$, then 
	$|y^{(d_1)}| + \cdots + |y^{(d_b)}| + |u_{t'}^{(j)}|_{s_j} \le s_{d_{i}}-s_j -nm$. 
\item[(iii)] 
$\dis{ v \wedge u_{t+nm}^{(j)} = q^{2bm} \sum_{t' , y} \alpha(t',y) \, u_{t'+nm}^{(j)} \wedge y \quad .}$
\item[(iv)] 
$\dis{
	u_{t}^{(j)} \wedge \overset{\leftarrow}{v}= \sum_{t' , y} \overline{\alpha(t',y)} \, \overset{\leftarrow}{y} \wedge u_{t'}^{(j)} 
	\quad , \quad 
	u_{t+nm}^{(j)} \wedge \overset{\leftarrow}{v}
	= q^{-2bm} \sum_{t' , y} \overline{\alpha(t',y)} \, \overset{\leftarrow}{y} \wedge u_{t'+nm}^{(j)} 
	\quad . }$
\item[(v)] 
$\overline{ B_{-m}'[j] \, ( v \wedge u_{t}^{(j)} )} = B_{-m}'[j] \, \overline{ v \wedge u_{t}^{(j)} } $. 

\end{itemize}
\label{corollary4-1}
\end{lemma}

\begin{proof}
For simplicity, we assume that $d_{1} = 1, d_{2} = 2, \cdots , d_{b} =b$. 

For all $1 \le i \le b$, we have 
\begin{align*}
	|\lambda^{(i)}| 
	&\le |\lambda^{(1)}| + \cdots + |\lambda^{(b)}| \le s_{i}-s_j -nm - |u_{t}^{(j)}|_{s_j} 
	=s_{i}-s_j -nm - t+s_j 
	\le s_i-t_0-nm
	=r_i-nm. 
\end{align*}
This proves (i). 

From Corollary \ref{property1} (i) (see also the remark after Definition \ref{def.6.1} ) and (i), we have 
\begin{align*}
	|y^{(1)}| + \cdots + |y^{(b)}| + |u_{t'}^{(j)}|_{s_j} 
	= |v^{(1)}| + \cdots + |v^{(b)}| + |u_{t'}^{(j)}|_{s_j} 
	= |\lambda^{(1)}| + \cdots + |\lambda^{(b)}| + |u_{t}^{(j)}|_{s_j} \le s_{i}-s_j -nm
\end{align*}
for all $1 \le i \le b$. 
This proves (ii). 

From (ii), we can apply Lemma \ref{lemma4} repeatedly and prove (iii).

Put $\overline{v \wedge u_{t}^{(j)}} = A(q) \, u_{t}^{(j)} \wedge \overset{\leftarrow}{v}$. 
Then, by the definition of bar involution (Definition \ref{barinv}) and Corollary \ref{property1} (iii), 
$$\overline{v \wedge u_{t+nm}^{(j)}} = A(q) \, u_{t+nm}^{(j)} \wedge \overset{\leftarrow}{v} \quad , \quad  
	\overline{y \wedge u_{t'}^{(j)}} = A(q) \, u_{t'}^{(j)} \wedge \overset{\leftarrow}{y} \quad \text{ and } \quad 
	\overline{y \wedge u_{t'+nm}^{(j)}} = A(q) \, u_{t'+nm}^{(j)} \wedge \overset{\leftarrow}{y}$$ 
	if $\alpha(t',y) \ne 0$. 
Thus, we obtain (iv) by taking bar involution of (\ref{eq.33}) and (iii). 

Finally, we prove (v). 
Put $\overline{v \wedge u_{t}^{(j)}} = A(q) \, u_{t}^{(j)} \wedge \overset{\leftarrow}{v}$. 
\begin{align*}
	B_{-m}'[j] \, \overline{v \wedge u_{t}^{(j)}}
	&= B_{-m}'[j] \, \left( A(q) \, u_{t}^{(j)} \wedge \overset{\leftarrow}{v} \right) \\
	&= A(q) \, B_{-m}'[j] \,\left( \sum_{t' , y} \overline{\alpha(t',y)} \, \overset{\leftarrow}{y} \wedge u_{t'}^{(j)} \right) 
		\hspace{3em} \text{(By (iv))} \\
	&= A(q) \, q^{-bm} \, \sum_{t' , y} \overline{\alpha(t',y)} \, \overset{\leftarrow}{y} \wedge (B_{-m}\,u_{t'}^{(j)}) 
		\hspace{3em} \text{(By Definition \ref{def.6.10})} \\
	&= A(q) \, q^{-bm} \, \sum_{t' , y} \overline{\alpha(t',y)} \, \overset{\leftarrow}{y} \wedge u_{t'+nm}^{(j)} .
\end{align*}
On the other hand, 
\begin{align*}
	\overline{B_{-m}'[j] \, \left( v \wedge u_{t}^{(j)} \right)}
	&= \overline{q^{-bm} \, v \wedge  \left( B_{-m} \, u_{t}^{(j)} \right) } 
	= \overline{q^{-bm} \, v \wedge u_{t+nm}^{(j)} } 
		\hspace{3em} \text{(By Definition \ref{def.6.10})} \\
	&= q^{bm} \, A(q) \, u_{t+nm}^{(j)} \wedge \overset{\leftarrow}{v} \\
	&= q^{bm} \, A(q) \, q^{-2bm} \left( \sum_{t' , y} \overline{\alpha(t',y)} \, \overset{\leftarrow}{y} \wedge u_{t'+nm}^{(j)} \right) 
		\hspace{3em} \text{(By (iv))} \\
	&= A(q) \, q^{-bm} \, \sum_{t' , y} \overline{\alpha(t',y)} \, \overset{\leftarrow}{y} \wedge u_{t'+nm}^{(j)} .
\end{align*}
\end{proof}

\begin{corollary}
Let $m \ge 0$ and $1 \le d_{1} < \cdots < d_{b} < j$. 
Let $N$ be a positive integer and $t_{0}$ be an integer such that $s_j - N \ge t_0$. 
Let $\lambda^{(d_1)} , \lambda^{(d_2)} ,\cdots , \lambda^{(d_b)} , \lambda^{(j)} \in \Pi$ such that 
$|\lambda^{(d_1)}| + \cdots + |\lambda^{(d_b)}| + |\lambda^{(j)}| \le s_{d_{i}}-s_j -nm$ for all $1 \le i \le b$. 
Put 
$$
w^{(j)} = v^{(j)}_{\lambda^{(j)} , N } 
	= u_{\lambda_1^{(j)}+s_{j}}^{(j)} \wedge u_{\lambda_2^{(j)}+s_{j}-1}^{(j)} \wedge \cdots 
	\wedge u_{\lambda_{N}^{(j)}+s_{j}-N+1}^{(j)} . 
$$
For $1 \le i \le b$, put $r_i = s_{d_i}-t_0$ and 
$$
	v^{(d_{i})} = v^{(d_{i})}_{\lambda^{(d_i)} , r_i} 
	= u_{\lambda_1^{(d_i)}+s_{d_i}}^{(d_i)} \wedge u_{\lambda_2^{(d_i)}+s_{d_i}-1}^{(d_i)} \wedge \cdots 
	\wedge u_{\lambda_{r_i}^{(d_i)}+s_{d_i}-r_i+1}^{(d_i)} . 
$$
Set $v = v^{(d_{1})} \wedge v^{(d_{2})} \wedge \cdots v^{(d_{b})}$. 
Then, 
$$ \overline{ B_{-m}'[j] \, ( v \wedge w^{(j)} )} = B_{-m}'[j] \, \overline{ v \wedge w^{(j)} } . $$ 
\label{corollary6}
\end{corollary}

\begin{proof}
We prove the assertion by induction on $N$. If $N=1$, then the assertion follows from Lemma \ref{corollary4-1} (v). 

Suppose that $N>1$. 
For $1 \le i \le N$, put $t_i=\lambda^{(j)}+s_j-i+1$ and 
$\tilde{w}^{(j)} = v^{(j)}_{\lambda^{(j)} , N-1 } = u_{t_1}^{(j)} \wedge u_{t_2}^{(j)} \wedge \cdots \wedge u_{t_{N-1}}^{(j)}$. 
Then, $t_i > t_0$ for any $1 \le i \le N$ and 
	$|\tilde{w}^{(j)}|_{s_j} = \lambda^{(j)}_1 + \cdots + \lambda^{(j)}_{N-1} \le |\lambda^{(j)}|$. 
By induction hypothesis, 
\begin{align}
	\overline{ B_{-m}'[j] \, ( v \wedge \tilde{w}^{(j)} )} = B_{-m}'[j] \, \overline{ v \wedge \tilde{w}^{(j)} } . 
\label{eq.34}
\end{align}
Put $\overline{v \wedge \tilde{w}^{(j)} \wedge u_{t_N}^{(j)}} = A(q) \, u_{t_N}^{(j)} \wedge \overline{v \wedge \tilde{w}^{(j)}}$ and 
$\overline{v \wedge \tilde{w}^{(j)}} = C(q) \, \tilde{w}^{(j)} \wedge \overset{\leftarrow}{v}$. 
Set the expansion of $v \wedge \tilde{w}^{(j)}$ as 
\begin{align}
	v \wedge \tilde{w}^{(j)} 
	&= \sum_{x,y^{(j)}} \alpha(x,y^{(j)}) \, y^{(j)} \wedge x \quad , 
\label{eq.35}
\end{align}
where $x = x^{(d_1)} \wedge \cdots \wedge x^{(d_b)}$ 
	(for all $1 \le i \le b$, $x^{(d_i)} = u_{\Bf{k}^{(d_i)}}^{(d_i)} \, (\Bf{k}^{(i)} \in \mathbb{Z}^{r_i}))$, 
	$y^{(j)} = u_{\Bf{k}}^{(j)} \, (\Bf{k} \in \mathbb{Z}^{N-1})$ and 
	$ \alpha(x,y^{(j)}) \in \mathbb{Q}(q) $. 
Then, 
\begin{align}
	\overline{ v \wedge \tilde{w}^{(j)} }
	&= C(q) \sum_{x,y^{(j)}} \overline{\alpha(x,y^{(j)})} \, \overset{\leftarrow}{x} \wedge \overset{\leftarrow}{y^{(j)}} . 
\label{eq.36}
\end{align}
Also set the expansion of $x \wedge u_{t_N}^{(j)}$ as 
\begin{align}
	x \wedge u_{t_N}^{(j)} 
	&= \sum_{t',z} \beta(t',z) \, u_{t'}^{(j)} \wedge z \quad , 
\label{eq.37}
\end{align}
where $z = z^{(d_1)} \wedge \cdots \wedge z^{(d_b)}$ 
	(for all $1 \le i \le b$, $z^{(d_i)} = u_{\Bf{g}^{(d_i)}}^{(d_i)} \, (\Bf{g}^{(i)} \in \mathbb{Z}^{r_i}))$ and 
	$ \beta(t',z) \in \mathbb{Q}(q) $. 
Then, from Lemma \ref{corollary4-1} (iv) 
\begin{align}
	u_{t_N}^{(j)} \wedge \overset{\leftarrow}{x} = \sum_{t',z} \overline{\beta(t',z)} \, \overset{\leftarrow}{z} \wedge u_{t'}^{(j)} 
	\quad , \quad 
	u_{t_N+nm}^{(j)} \wedge \overset{\leftarrow}{x} 
		= q^{-2bm} \sum_{t',z} \overline{\beta(t',z)} \, \overset{\leftarrow}{z} \wedge u_{t'+nm}^{(j)}  
	\quad . 
\label{eq.38}
\end{align}

Under the above preparation, 
\begin{align*}
	&\overline{ B_{-m}'[j] \, ( v \wedge w^{(j)} )} \\
	&= \overline{ B_{-m}'[j] \, ( v \wedge \tilde{w}^{(j)} \wedge u_{t_N}^{(j)} )} \\
	&= \overline{ q^{-bm} \, v \wedge \left( B_{-m} \tilde{w}^{(j)} \right) \wedge u_{t_N}^{(j)} }
		+ \overline{ q^{-bm} \, v \wedge \tilde{w}^{(j)} \wedge u_{t_N+nm}^{(j)} } \\
	&= q^{bm}A(q) \, u_{t_N}^{(j)} \wedge \overline{v \wedge \left( B_{-m} \tilde{w}^{(j)} \right)} 
		+ q^{bm}A(q) \, u_{t_N+nm}^{(j)} \wedge \overline{v \wedge \tilde{w}^{(j)}} \\
	&= A(q) \, u_{t_N}^{(j)} \wedge \left( B_{-m}'[j] \overline{v \wedge \tilde{w}^{(j)} } \right) 
		+ q^{bm}A(q) \, u_{t_N+nm}^{(j)} \wedge \overline{v \wedge \tilde{w}^{(j)}} 
			\hspace{2em} \text{ (By (\ref{eq.34})) } \\
	&= A(q) C(q) \, u_{t_N}^{(j)} \wedge 
		B_{-m}'[j] \left( \sum_{x,y^{(j)}} \overline{\alpha(x,y^{(j)})} \, \overset{\leftarrow}{x} \wedge \overset{\leftarrow}{y^{(j)}} \right) 
		+ q^{bm} A(q) C(q) \, u_{t_N+nm}^{(j)} \wedge 
		\left( \sum_{x,y^{(j)}} \overline{\alpha(x,y^{(j)})} \, \overset{\leftarrow}{x} \wedge \overset{\leftarrow}{y^{(j)}} \right) \\ 	&		\hspace{10em} \text{ (By (\ref{eq.36})) } \\
	&= q^{-bm} A(q) C(q) \sum_{x,y^{(j)}} \overline{\alpha(x,y^{(j)})} \, 
		u_{t_N}^{(j)} \wedge \overset{\leftarrow}{x} \wedge \left( B_{-m} \overset{\leftarrow}{y^{(j)}} \right)
		+ q^{bm} A(q) C(q) \sum_{x,y^{(j)}} \overline{\alpha(x,y^{(j)})} \, 
		u_{t_N+nm}^{(j)} \wedge \overset{\leftarrow}{x} \wedge \overset{\leftarrow}{y^{(j)}} \\
	&= q^{-bm} A(q) C(q) \sum_{x,y^{(j)}} \sum_{t',z} \overline{\alpha(x,y^{(j)})} \overline{\beta(t',z)} \, 
		\overset{\leftarrow}{z} \wedge u_{t'}^{(j)} \wedge \left( B_{-m} \overset{\leftarrow}{y^{(j)}} \right) \\
	&	\hspace{3em} + q^{-bm} A(q) C(q) \sum_{x,y^{(j)}} \sum_{t',z} \overline{\alpha(x,y^{(j)})} \overline{\beta(t',z)} \, 
		\overset{\leftarrow}{z} \wedge u_{t'+nm}^{(j)} \wedge \overset{\leftarrow}{y^{(j)}} 
		\hspace{3em} \text{ (By (\ref{eq.38})) } \\
	&= A(q) C(q) \sum_{x,y^{(j)}} \sum_{t',z} \overline{\alpha(x,y^{(j)})} \overline{\beta(t',z)} \, 
		B_{-m}'[j] \, \left( \overset{\leftarrow}{z} \wedge u_{t'}^{(j)} \wedge \overset{\leftarrow}{y^{(j)}} \right) .
\end{align*}
On the other hand, 
\begin{align*}
	&B_{-m}'[j] \, \overline{v \wedge w^{(j)}} \\
	&= B_{-m}'[j] \, \overline{ v \wedge \tilde{w}^{(j)} \wedge u_{t_N}^{(j)} } \\
	&= B_{-m}'[j] \, \left( A(q) \, u_{t_N}^{(j)} \wedge \overline{ v \wedge \tilde{w}^{(j)} } \right) \\
	&= A(q) C(q) B_{-m}'[j] \, \left( \sum_{x,y^{(j)}} \overline{\alpha(x,y^{(j)})} \, 
		u_{t_N}^{(j)} \wedge \overset{\leftarrow}{x} \wedge \overset{\leftarrow}{y^{(j)}} \right) 
		\hspace{3em} \text{ (By (\ref{eq.36})) } \\
	&= A(q) C(q) \sum_{x,y^{(j)}} \sum_{t',z} \overline{\alpha(x,y^{(j)})} \overline{\beta(t',z)} \, 
		B_{-m}'[j] \, \left( \overset{\leftarrow}{z} \wedge u_{t'}^{(j)} \wedge \overset{\leftarrow}{y^{(j)}} \right) 
		\hspace{3em} \text{ (By (\ref{eq.38})) } .
\end{align*}
\end{proof}

\begin{corollary}
Let $m \ge 0$, $1 \le j \le \ell$, $t_0 \in \mathbb{Z}$ and $\Bf{\lambda} \in \Pi^{\ell}$. 
Suppose that $| \Bf{\lambda} ; \Bf{s} \rangle$ is $nm$-dominant and $t_0 \le s_{\ell} - l(\lambda^{(\ell)}) $.
For $1 \le i \le \ell$, put $r_i = s_{i}-t_0$ and 
$$
	v^{(i)} = v^{(i)}_{\lambda^{(i)} , r_i} 
	= u_{\lambda_1^{(i)}+s_{i}}^{(i)} \wedge u_{\lambda_2^{(i)}+s_{i}-1}^{(i)} \wedge \cdots 
	\wedge u_{\lambda_{r_i}^{(i)}+s_{i}-r_i+1}^{(i)} . 
$$
Set $v = v^{(1)} \wedge v^{(2)} \wedge \cdots v^{(j-1)}$ and 
	$w = v^{(j)} \wedge v^{(j+1)} \wedge \cdots v^{(\ell)}$. 
Then, 
$$ \overline{ B_{-m}[j,\ell] \, ( v \wedge w )} = B_{-m}[j,\ell] \, \overline{ v \wedge w } . $$ 
\label{corollary7}
\end{corollary}

\begin{proof}
By applying Corollary \ref{corollary6} repeatedly, we obtain the assertion. 
\end{proof}

\begin{proof}[Proof of Proposition \ref{mainconj}]
The statement (ii) follows from Corollary \ref{corollary7}. 
We prove (i). 
For convenience, we define $B_{-m}[\ell+1,\ell] = B_{-m}[\ell+2,\ell] = 0$ and $B_{-m}'[\ell+1]=0$.

From $B_{-m}[j,\ell] =\sum_{i=j}^{\ell} q^{(i-j)m} B_{-m}'[i]$, we have 
\begin{align*}
	B_{-m}'[j] 
	&= B_{-m}[j,\ell] - \sum_{i=j+1}^{\ell} q^{(i-j)m} B_{-m}'[i] \\
	&= B_{-m}[j,\ell] - q^{m} \sum_{i=j+1}^{\ell} q^{(i-j-1)m} B_{-m}'[i] \\
	&= B_{-m}[j,\ell] - q^{m} B_{-m}[j+1,\ell] \quad .
\end{align*}
Thus, 
\begin{align*}
	B_{-m}[j] 
	&= B_{-m}'[j] -q^{-m} B_{-m}'[j+1] \\
	&= B_{-m}[j,\ell] - q^{m} B_{-m}[j+1,\ell] -q^{-m} B_{-m}[j+1,\ell] + B_{-m}[j+2,\ell] \\
	&= B_{-m}[j,\ell] - (q^{m}+q^{-m}) B_{-m}[j+1,\ell] + B_{-m}[j+2,\ell] \quad .
\end{align*}
Hence from (ii), 
\begin{align*}
	\overline{ B_{-m}[j] \, u } 
	&= \overline{ B_{-m}[j,\ell] \, u } - (q^{m}+q^{-m}) \overline{ B_{-m}[j+1,\ell] \, u } + \overline{ B_{-m}[j+2,\ell] \, u } \\
	&= B_{-m}[j,\ell] \, \overline{u} - (q^{m}+q^{-m}) B_{-m}[j+1,\ell] \, \overline{u} + B_{-m}[j+2,\ell] \, \overline{u} 
	= B_{-m}[j] \, \overline{u} \quad .
\end{align*}

\end{proof}

%
%



\end{document}